\documentclass{amsart}
\usepackage{graphicx}
\usepackage{dsfont}
\usepackage{amssymb}
\usepackage[table]{xcolor}
\usepackage{enumerate}
 \usepackage[normalem]{ulem}

\newcommand{\MO}[1]{{\color{black} #1}}
\newtheorem{proposition}{Proposition}
\newtheorem{lemma}[proposition]{Lemma}

\newtheorem{theorem}[proposition]{Theorem}

\theoremstyle{definition}
\newtheorem{remark}[proposition]{Remark}
\newtheorem{definition}[proposition]{Definition}
\newtheorem{example}[proposition]{Example}

\newcommand{\II}{\mathds{I}}
\newcommand{\PP}{\mathds{P}}
\newcommand{\RR}{\mathds{R}}

\newcommand{\ns}[1]{\textcolor{black}{#1}}

\newcommand{\cc}{\cellcolor{black!10!white}}
\begin{document}

\title{Freedom in constructing quasi-copulas vs.\ copulas }%
\author{Matja\v{z} Omladi\v{c}}%
\address{Institute of Mathematics, Physics and Mechanics, Ljubljana, Slovenia}%
\email{matjaz@omladic.net}%
\author{Nik Stopar}%
\address{University of Ljubljana, Faculty of Civil and Geodetic Engineering, University of Ljubljana, Faculty of Mathematics and Physics, and Institute of Mathematics, Physics and Mechanics, Ljubljana, Slovenia}%
\email{nik.stopar@fgg.uni-lj.si}%

\thanks{The first author acknowledges financial support from the Slovenian Research Agency (research core funding No. P1-0448). The second author acknowledges financial support from the Slovenian Research Agency (research core funding No. P1-0222 and project No. J1-50002.)}

\begin{abstract}
\ns{The main goal of this paper is to study the extent of freedom one has in constructing quasi-copulas vs.\ copulas. Specifically, it exhibits} 
three construction methods for quasi-copulas based
on recent developments: a representation of multivariate quasi-copulas by
means of infima and suprema of copulas, an extension of a classical result on
shuffles of min to the setting of quasi-copulas, and a construction method for
quasi-copulas obeying a given signed mass pattern on a patch.

\end{abstract}
\maketitle

\section{Introduction}\label{sec:intro}

Copulas have become an important theoretical notion and a very powerful practical tool in many applications. They are simply distribution functions with uniform marginal distributions. 
However, every distribution function can be expressed via a copula and its margins (Sklar \cite{Skla}).

 Quasi-copulas may be viewed as pointwise infima and suprema of copulas (their exact definition will be given later).
While questions on copulas are attracting more and more attention, there are many interesting and important problems in quasi-copulas still unanswered.
An excellent list of problems in multivariate quasi-copulas, \MO{sometimes popularly} called the "Hitchhiker’s", is given in \MO{\cite[Section 9]{ArGaMeDeBa}}. As a matter of fact, the result of \cite{Klemetal} we are using here may be seen as a solution to the Hitchhiker’s problem 2.
\ns{
Actually, it is shown in \cite{Klemetal} that for any two points $\mathbf{u}$ and $\mathbf{v}$ and any quasi-copula $Q$ there is a copula $C$ that attains the same values at $\mathbf{u}$ and $\mathbf{v}$ as $Q$.
Another two 
of these problems are closely related to our work presented here, namely problem 6 (cf.\ explanation at the beginning of Section \ref{sec:patch+-}) and problem 3 whose brief discussion follows. }

In the bivariate case the Dedekind-MacNeille completion of the poset of bivariate copulas (w.r.t. the pointwise order) is just the class of bivariate quasi-copulas \cite{NeUbFl2}.
In the multivariate setting with dimension $d \geqslant 3$ this is not true \cite{FeSaNeUbFl}.
There exist at least two copies of the Dedekind-MacNeille completion of multivariate copulas inside the family of multivariate quasi-copulas of the same dimension, as was shown recently \cite{OmSt5} \ns{(thus solving problem 3)}.
These two families are relatively very small, and the lattice operations on the completions extend but are not equal to the starting operations of pointwise infimum and supremum.
This raises the question whether the family of quasi-copulas is too big in higher dimensions.
One of the \MO{
aims} of this paper is to answer this question in the negative at least for the pointwise order. Our answer uses an even more recent result on copulas with given values at two points \cite{Klemetal} \ns{(linked to Hitchhiker's problem 2 as mentioned above)}.

The recent development in the area includes: extensions of Sklar’s theorem from the setting of distributions and copulas to the setting of quasi-distributions and quasi-copulas \cite{OmSt2,OmSt4}, measure theoretic properties of quasi-copulas \cite{NeQuMoRoLaUbFl1,DeBaMeUbFl,St,NeQuMoRoLaUbFl3,FeSaRoLaUbFl}, approximations with quasi-copulas and related topological properties \cite{NeQuMoRoLaUbFl1,DuFeSaTr,Tru}, and patchwork constructions of quasi-copulas \cite{QuMoSe,DeBaDeMe,KoBuKoOmSt}.
Perhaps the most important among these results from the point of view of applications are constructions of multivariate quasi-copulas, since they are helping us understand these objects in a more perceptible way.

\ns{The main goal of this paper is to study the extent of freedom one has in constructing quasi-copulas vs.\ copulas. On one hand there is the amazing freedom in constructing quasi-copulas with respect to construction of negative volumes both in magnitude and distribution (cf.\ Section \ref{sec:patch+-}). On the other hand, some of the strongest copula construction methods, s.a.\ shuffles, when brought (for the first time) to the environment of quasi-copulas collide importantly with their structure, so that only a sophisticated approach can bring us to beneficial results. We exhibit here some additional specific aims. 
} First, 
the construction method developed in \cite{Klemetal} to solve the Hitchhiker’s problem~2 helps us give a much better feeling on how big the family of multivariate quasi-copulas is compared to the family of copulas of the same dimension in the sense of pointwise order. \ns{
We show that every quasi-copula can be obtained as a pointwise supremum of a set of pointwise infima of sets of copulas. It can also be obtained as a pointwise infimum of a set of pointwise suprema of sets of copulas.}
Next, we deepen this result by combining it with one of the most important source of constructing copulas, shuffles. Our method is nontrivial since the shuffle of a quasi-copula is not always a quasi-copula unlike in the copula case where their shuffles are necessarily copulas.
To make the presentation clearer, we first stage the bivariate approach. When extending it to the multivariate case, we need an additional notion of tensor products of discrete quasi-copulas.
This concept seems to be new and of independent interest. The third construction produces a quasi-copula with a given signed pattern of masses on a patch. The construction method extends a calculation of a concrete counterexample in the theory of imprecise copulas, i.e., an example of a (discrete) imprecise copula that does not avoid sure loss, see \cite{OmSt1}.
The paper exhibits the first of the methods mentioned above in Section 2, the second one in Sections 3 and 4, and the third one in Section 5.

\section{Lattice-theoretic properties of multivariate quasi-copulas}\label{sec:mot}

Although the reader is assumed familiar with basic copula theory \cite{DuSe,Nels}, we
will first give the definitions of copulas and quasi-copulas, and set some notation
on the way. Let $\II = [0, 1]$. For two points $\mathbf{x}, \mathbf{y} \in \RR^d$ we write $\mathbf{x} \leqslant \mathbf{y}$ if $x_i \leqslant y_i$ for every $i \in [d]$, where $[d] = \{1,2,\ldots,d\}$. If $\mathbf{x} \leqslant \mathbf{y}$, a \emph{$d$-box} (rectangle) is the set
$$[\mathbf{x}, \mathbf{y}] = [x_1, y_1] \times \ldots \times [x_d, y_d] \subseteq \RR^d.$$
Denote also by $\textup{ver} [\mathbf{x}, \mathbf{y}] = \{x_1, y_1\} \times \ldots \times \{x_d, y_d\}$
the \emph{set of vertices} of the box $[\mathbf{x}, \mathbf{y}]$.
For a real valued function $A$ and  $d$-box $R = \MO{ [\mathbf{x}, \mathbf{y}]}$ we define the \emph{$A$-volume} (or simply the \emph{volume} if $A$ is understood) of the box $R$ by
$$V_A(R) =\sum_{v \in \textup{ver} R} \textup{sign}_R(\mathbf{v}) A(\mathbf{v}),$$
where $\textup{sign}_R(\mathbf{v})$ equals $1$ if $v_j = x_j$ for an even number of indices, and $-1$ otherwise. A $d$-variate \emph{quasi-copula} is a function
$A \colon \II^d \to \II$ that satisfies the following three conditions:
\begin{enumerate}[$(i)$]
\item for every $j \in [d]$ we have $A(1,\ldots, 1, u_j, 1,\ldots, 1) = u_j$,
\item $A$ is increasing in each of its variables, \MO{ i.e., for every $j\in[d]$ and for every point $\mathbf{u}_j=(u_1,\ldots,u_{j-1}, u_{j+1}, \ldots,u_d)\in \II^{d-1}$ the function ${f}_j:t\mapsto A(u_1,\ldots,u_{j-1},t, u_{j+1}, \ldots,u_d)$ is increasing;}
\item $A$ satisfies the 1-Lipschitz condition, \MO{ i.e., if $\mathbf{u},\mathbf{v} \in\II^d$, then
      \[
        |A(\mathbf{v})-A(\mathbf{u})|\leqslant\sum_{j=1}^{n}|v_j-u_j|.
      \]}
\end{enumerate}
\MO{ Condition $(iii)$ is sometimes replaced by an equivalent condition
\begin{enumerate}[$(i)$]
\item[$(iii')$] $A$ has the 1-Lipschitz property in each direction, i.e., for every $j\in[d]$ and every $\mathbf{u}_j$ as in point $(ii)$, the there defined function $f_j$ is 1-Lipschitz meaning that it satisfies $|{f}_j(u)-{f}_j(v)|\leqslant|u-v|$ for all $u,v\in\II$;
\end{enumerate}
(We prefer to use this form of the condition in the sequel.)}
It is well-known and not hard to see that $A$ then satisfies also 
\begin{enumerate}[$(i)$]
\item[$(iv)$] for every $j \in [d]$ and for every point $(u_1,\ldots, u_{j-1}, u_{j+1},\ldots, u_d) \in \II^{d-1}$ we have $A(u_1,\ldots, u_{j-1}, 0, u_{j+1},\ldots, u_d) = 0$.
\end{enumerate}
Function $A$ is a \emph{copula} if it satisfies Conditions $(i)$, $(iv)$, and
\begin{enumerate}[$(i)$]
\item[$(ii')$] for every rectangle $R \subseteq \II^d$ its $A$-volume is nonnegative.
\end{enumerate}
\MO{In Sections \ref{sec:gen} and \ref{sec:patch+-} we need the notion of \emph{discrete (quasi-)copula}; this will mean, given a finite mesh $\Delta\subseteq \II^d$, simply a (quasi-)copula restricted to the subset $\Delta$.
}\\

Here is a simple observation, which is, nevertheless, a key to our first result.

\begin{proposition}\label{prop:source}
    Let $\mathcal{D}$ be a family of copulas such that for every quasi-copula $Q$ and every two points $\mathbf{x},\mathbf{z}\in\II^d$ there exists $D_{\mathbf{x},\mathbf{z}}^Q\in\mathcal{D}$ such that $D_{\mathbf{x},\mathbf{z}}^Q(\mathbf{x})=Q(\mathbf{x})$ and $D_{\mathbf{x},\mathbf{z}}^Q(\mathbf{z})=Q(\mathbf{z})$. Then
  \[
    Q(\mathbf{u})=\inf_{\mathbf{z}\in\II^d}\left(\sup_{\mathbf{x}\in\II^d} D_{\mathbf{x},\mathbf{z}}^Q(\mathbf{u})\right) \quad \mbox{and}\quad Q(\mathbf{u})=\sup_{\mathbf{z}\in\II^d}\left(\inf_{\mathbf{x}\in\II^d} D_{\mathbf{x},\mathbf{z}}^Q(\mathbf{u})\right)
  \]
  for all $\mathbf{u} \in \II^d$.
\end{proposition}

\begin{proof}
 For any $\mathbf{z} \in \II^d$ the function $A_\mathbf{z}^Q(\mathbf{u})=\sup_{\mathbf{x}\in\II^d} D_{\mathbf{x},\mathbf{z}}^Q(\mathbf{u})$ is a quasi-copula that satisfies $A_\mathbf{z}^Q(\mathbf{z})=Q(\mathbf{z})$ and $A_\mathbf{z}^Q(\mathbf{u}) \geq Q(\mathbf{u})$ for any other point $\mathbf{u} \in \II^d$.
Hence, $\inf_{\mathbf{z}\in\II^d} A_\mathbf{z}^Q(\mathbf{u})$ has to be equal to $Q(\mathbf{u})$ for all $\mathbf{u} \in \II^d$. This proves the first of the two equalities and a similar argument works for the second one.
\end{proof}

It was proved recently in \cite[Proposition~8.6]{Klemetal} that the family of all copulas satisfies the condition on family $\mathcal{D}$ in Proposition~\ref{prop:source}, which immediately implies the following.

\begin{theorem}\label{thm:copulas}
Every \MO{quasi-copula} can be obtained as a pointwise supremum of a set of pointwise infima of sets of copulas. It can also be obtained as a pointwise infimum of a set of pointwise suprema of sets of copulas. So, for every quasi-copula $Q$ there exist families of copulas $\{\mathcal{C}_i\}_{i\in I}$ and $\{\mathcal{\widetilde{C}}_j\}_{j\in J}$ such that
\[
    Q(\mathbf{u})=\inf_{i\in I}\sup_{C\in\mathcal{C}_i}C(\mathbf{u})\quad\mbox{and}\quad Q(\mathbf{u})=\sup_{j\in J}\inf_{C\in\mathcal{\widetilde{C}}_j}C(\mathbf{u}).
\]
\end{theorem}

In other words, $d$-variate quasi-copulas for $d \geq 3$ can be obtained from $d$-variate copulas via the application of pointwise infima and suprema in two steps.
In contrast, for $d=2$ only one step is needed, since any bivariate quasi-copula is simultaneously an infimum of bivariate copulas and a supremum of bivariate copulas \cite{NeUbFl2}.

Together with the results of \cite{OmSt5}, this finally explains the role of quasi-copulas in the study of copulas when $d \geqslant 3$.
If one wishes to work with the Dedekind-MacNeille completion of copulas, then not all quasi-copulas are needed, since the completion is a much smaller set. However, although the order in the completion is the pointwise order, the lattice operations are not the pointwise infimum and supremum.
On the other hand, if one insists on working with the operations of pointwise infimum and supremum, as is usually the case, then all quasi-copulas are needed.

Next we strengthen the above result by showing that the same holds if we replace the set of copulas with the set of shuffles of min. Again only two steps are needed to generate all quasi-copulas.
\MO{The shuffles of min were first introduced by Mikusi\'{n}ski et al. \cite{MiShTa} in the bivariate case and then extended to the multivariate case in Mikusi\'{n}ski and Taylor \cite{MiTa} (see also \cite{DuFeSa} for a slightly more general definition, which we will not need here). Observe that in the bivariate case two types of shuffles were introduced, straight and flip, but only the first one of the two was extended to the multivariate case. So, we will use the term ($d$-variate) shuffle of min in all dimensions meaning straight shuffle of min in case $d=2$.}

For the formal definitions we refer the reader to the aforementioned papers, but informally a shuffle of min can be described in the following way.
We start with the unit cube $\II^d$ with mass distributed according to copula $M$, i.e., mass is uniformly distributed on the main diagonal\MO{, so that} $x_1=x_2=\ldots=x_d$.
Then we slice the cube with a finite number of hyperplanes of the form $x_1=c$, rearrange the slices in a different order by means of a permutation, and recombine them to obtain a copy of $\II^d$, in which the mass is now redistributed.
We repeat this process for variables $x_2,\ldots,x_d$, possibly choosing different numbers of hyperplanes and a different permutation each time. We end up with $\II^d$ containing a redistribution of the mass of copula $M$. The copula that has mass distributed in such a way is called a \emph{shuffle of min}.\\[2mm]

\ns{
Let $(\Xi(\II^d),\|.\|_\infty)$ be the space of all real-valued continuous functions defined on $\II^d$ equipped with the supremum norm $\|f\|_\infty=\sup_{\mathbf{u} \in \II^d} |f(\mathbf{u})|$.
It is well known that the set of all $d$-variate copulas \MO{and} the set of all $d$-variate quasi-copulas are compact subsets of $(\Xi(\II^d),\|.\|_\infty)$. Furthermore, the following result was proved in \cite{MiTa}.}

\ns{
\begin{theorem}\label{thm:miku}
The set of all shuffles of min is a dense subset of the set of all $d$-variate copulas.
\end{theorem}}

\MO{By modifying the construction from \cite[\S 6]{MiTa} we obtain the following result which shows that we can interpolate any finite set of values of a copula with a shuffle of min. }

\begin{proposition}\label{prop:som}
For any copula $C$ and any finite set of points $S \subseteq \II^d$ there exists a 
shuffle of min $\MO{\widetilde{C}}$ such that $\MO{\widetilde{C}}(\mathbf{u})=C(\mathbf{u})$ for all $\mathbf{u} \in S$.
\end{proposition}

\begin{proof}
For every $k\in [d]$ let $S_k$ be the set of $k$-th coordinates of points in $S$. By adding some elements to the sets $S_k$, if needed, we may assume that all of them contain $0$ and $1$ and are of the same size $m+1$.
Denote $S_k=\{x_0^k,x_1^k,\ldots,x_m^k\}$, where $0=x_0^k<x_1^k<\ldots<x_m^k=1$.
For each $d$-tuple $\mathbf{i}=(i_1,i_2,\ldots,i_d)\in [m]^d$ let $R_\mathbf{i}=[x_{i_1-1}^1,x_{i_1}^1] \times [x_{i_2-1}^2,x_{i_2}^2]\times \ldots \times [x_{i_d-1}^d,x_{i_d}^d]$ and $\mu_\mathbf{i}=\mu(R_\mathbf{i})$ where $\mu$ is the probability measure induced by copula $C$.
Denote the lexicographic order on $d$-tuples by $<_\text{lex}$. For any $k \in [d]$ and any $d$-tuple $\mathbf{i} \in [m]^d$ define
$$q_\mathbf{i}^k=x_{i_k-1}^k+\sum_{\substack{\mathbf{j}<_\text{lex}\mathbf{i}\\j_k=i_k}} \mu_\mathbf{j}.$$
Note that the points $q_\mathbf{i}^k$ with $i_k=p$ divide the interval $[x_{p-1}^k,x_p^k]$ into subintervals
$J_\mathbf{i}^k=[q_\mathbf{i}^k,q_\mathbf{i}^k+\mu_\mathbf{i}]$ with $i_k=p$. For fixed $i$ the length of $J_\mathbf{i}^k$ is equal to $\mu_\mathbf{i}$ and is independent of $k$.
Hence, for each $\mathbf{i} \in [m]^d$, we may distribute mass $\mu_\mathbf{i}$ uniformly along the main diagonal of the $d$-dimensional cube $J_\mathbf{i}^1 \times J_\mathbf{i}^2 \times\ldots\times J_\mathbf{i}^d$. We denote the grounded function that distributes mass in such a way by $\MO{\widetilde{C}}$.
For every $k \in [d]$, the union of the collection of intervals $\{J_\mathbf{i}^k ~|~ \mathbf{i}\in [m]^d\}$ is equal to $\II$ and two intervals in the collection intersect in at most one point.
This implies that $\MO{\widetilde{C}}$ is a shuffle of min.
Clearly, the only mass that $\MO{\widetilde{C}}$ distributes inside $d$-box $R_\mathbf{i}$ is the mass $\mu_\mathbf{i}$ in its subbox $J_\mathbf{i}^1 \times J_\mathbf{i}^2 \times\ldots\times J_\mathbf{i}^d$, where $\mathbf{i}=(i_1,i_2,\ldots,i_d)$.
Consequently, $V_C(R_\mathbf{i})=\mu_\mathbf{i}=V_\MO{\widetilde{C}}(R_\mathbf{i})$ and therefore copulas $C$ and $\MO{\widetilde{C}}$ have equal values in all the vertices of the boxes $R_\mathbf{i}$, $\mathbf{i}\in [m]^d$.
In particular, $C$ and $\MO{\widetilde{C}}$ coincide on the set $S$.
\end{proof}

\MO{
Note that Proposition~\ref{prop:som} is a strengthening of Theorem~\ref{thm:miku}. Indeed, let us fix an integer $n$ and take $S$ to be a finite mesh $S=\left\{0,\dfrac{1}{n},\dfrac{2}{n},\ldots,1\right\}^d$. Now, if copula $\widetilde{C}$ interpolates copula $C$ on $S$, then the $1$-Lipschitz condition of copulas implies $|\widetilde{C}(\mathbf{u})-C(\mathbf{u})| \leqslant \dfrac{2d}{n}$ for all $\mathbf{u}\in\II^d$, so that $\widetilde{C}$ approximates $C$ uniformly on $\II^d$.}
With the help of Proposition~\ref{prop:som} we can now strengthen the result of Theorem~\ref{thm:copulas}.

\begin{theorem}\label{thm:shuffles}
Every quasi-copla can be obtained as a pointwise supremum of a set of pointwise infima of sets of 
shuffles of min. It can also be obtained as a pointwise infimum of a set of pointwise suprema of sets of 
shuffles of min. So, for every quasi-copula $Q$ there exist families of shuffles of min $\{\mathcal{M}_i\}_{i\in I}$ and $\{\mathcal{\widetilde{M}}_j\}_{j\in J}$ such that
\[
    Q(\mathbf{u})=\inf_{i\in I}\sup_{C\in\mathcal{M}_i}C(\mathbf{u})\quad\mbox{and}\quad Q(\mathbf{u})=\sup_{j\in J}\inf_{C\in\mathcal{\widetilde{M}}_j}C(\mathbf{u}).
\]
\end{theorem}

\begin{proof}
Let $Q$ be an arbitrary quasi-copula and $\mathbf{x}$ and $\mathbf{z}$ arbitrary points in $\II^d$. By \cite[Proposition~8.6]{Klemetal} there exists a copula $C$ such that $C(\mathbf{x})=Q(\mathbf{x})$ and $C(\mathbf{z})=Q(\mathbf{z})$.
Let $S=\{\mathbf{x},\mathbf{z}\}$. Then by Proposition~\ref{prop:som} there exists a shuffle of min $\MO{\widetilde{C}}$ such that $\MO{\widetilde{C}}(\mathbf{x})=C(\mathbf{x})=Q(\mathbf{x})$ and $\MO{\widetilde{C}}(\mathbf{z})=C(\mathbf{z})=Q(\mathbf{z})$.
This shows that the family of all shuffles of min satisfies the condition on family $\mathcal{D}$ of Proposition~\ref{prop:source} and a direct application of the proposition gives the desired result.
\end{proof}

We remark that the above proof actually shows that the complete lattice of quasi-copulas is generated by an even smaller set of $3$-slice shuffles of min, i.e., shuffles obtained by splitting $\II^d$ into $3$ slices in each of the coordinates, which means that $\II^d$ is divided into $3^d$ boxes altogether.

\section{Shuffles of quasi-copulas and bivariate approximation}\label{sec:shuff}

In this section we investigate approximations of quasi-copulas with shuffles. \MO{However, let us first start by a brief overview of shuffles of copulas.}
\ns{
The motivation for introducing and studying shuffles of copulas goes back to Vitale's seminal paper \cite{Vi}.
The first formal definition of a shuffle was given by Mikusi\' nski et al. \cite{MiShTa}, who introduced shuffles of min in the bivariate setting, by shuffling the mass distribution of $M$ in the $x$-direction.
They also gave a stochastic interpretation of such shuffles in the lines of Vitale's work.
Mikusi\' nski and Taylor later extended the original definition to the multivariate setting \cite{MiTa}, where shuffling is performed in the directions of all coordinate axis simultaneously.
Durante et al. \cite{DuSaSe} interpreted bivariate shuffles of min using interval exchange transformations and the push-forward of the measure induced by $M$.
This allowed them to extend the existing definition by replacing interval exchange transformations with general Lebesgue-measure-preserving transformations and replacing copula $M$ with an arbitrary copula $C$.
Trutschnig and Fern\'{a}ndez S\'{a}nchez \cite{TrFeSa2} established a connection between this type of shuffles and Markov operators, and used it to study iterative properties of shuffles.
A further extension of the above definition that includes shuffling in both directions was considered recently by Griessenberger et al. \cite{GrFeSaTr}, who investigated the properties of shuffles with pairs of L\"{u}roth maps.
}

\ns{
Theorems \ref{thm:copulas} and \ref{thm:shuffles} suggest a possibility of an indirect stochastic interpretation for quasi-copulas, since they are infima and suprema of sets of copulas.
This opens an option of shuffling quasi-copulas indirectly by shuffling the corresponding sets of copulas as follows. For simplicity we limit our considerations to the bivariate case and to shuffling in $x$-direction only.
Fix a shuffling mechanism $S$ by choosing a partition $J_1,J_2,\ldots,J_n$ of $(0,1]$ and a permutation $\pi$ of the index set $\{1,2,\ldots,n\}$. Here $J_i=(x_{i-1},x_i]$ for some $0=x_0<x_2< \ldots <x_n=1$.
For any copula $C$ let $C^S$ denote the shuffle of $C$ by means of $S$ in $x$ direction, i.e., we shuffle the vertical strips $J_i \times \II$ with permutation $\pi$ into a new order $J_{\pi(1)},J_{\pi(2)},\ldots,J_{\pi(n)}$.
Now suppose $Q(u,v)=\sup_{j \in J} C_j(u,v)$ for all $u,v \in \II^2$, where $J$ is some index set and $C_j$ are copulas. Then $\sup_{j \in J} C_j^S(u,v)$ is another quasi-copula which could be declared to be an $S$-shuffle of $Q$.
Unfortunately, this kind of shuffles are not well defined since different representations of $Q$ produce different results as demonstrated in the next example.}

\ns{
\begin{example}\label{ex:sh1}
Let $S$ be a shuffling mechanism given by the partition $J_1=(0,\frac13], J_2=(\frac13,\frac23], J_3=(\frac23,1]$ and the cyclic permutation $\pi=(1,2,3)$.
Let $C_1, C_2, C_3, C_4$ be singular copulas with all of their mass distributed on the graphs of functions
\begin{align*}
f_1(u) &=
\begin{cases}
u+\frac23, & 0\leqslant u \leqslant \frac13,\\
u-\frac13, & \frac13 <u \leqslant 1,
\end{cases}
&
f_2(u) &=
\begin{cases}
u+\frac13, & 0\leqslant u \leqslant \frac23,\\
u-\frac23, & \frac23 <u \leqslant 1,
\end{cases}\\
f_3(u) &=
\begin{cases}
1-u, & 0\leqslant u \leqslant \frac13,\\
u-\frac13, & \frac13 <u \leqslant 1,
\end{cases}
&
f_4(u) &=
\begin{cases}
u+\frac13, & 0\leqslant u \leqslant \frac23,\\
1-u, & \frac23 <u \leqslant 1.
\end{cases}
\end{align*}
The mass distribution of copulas $C_1, C_2, C_3, C_4$ and their shuffles $C_1^S, C_2^S, C_3^S, C_4^S$ is shown in Figure~\ref{fig:sh1}. Here, the point-wise supremum operation is denoted by $\vee$ for simplicity.
Note that $C_1 \vee C_2=C_3 \vee C_4$ (denoted by $Q$ in Figure~\ref{fig:sh1}), however, $C_1^S \vee C_2^S=M$ while $C_3^S \vee C_4^S \neq M$.
\begin{figure}[!ht]
\centering
\includegraphics[width=\textwidth]{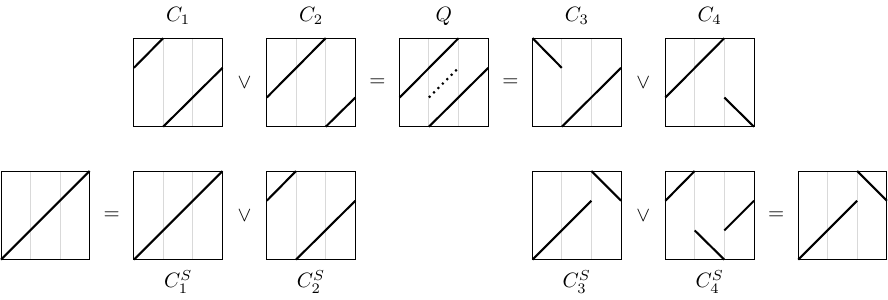}
\caption{Mass distributions of copulas $C_1,C_2,C_3,C_4$ and their shuffles from Example~\ref{ex:sh1}, along with the mass distributions of some of their supremums. The full lines indicate positive mass while the dotted line indicates negative mass.}
\label{fig:sh1}
\end{figure}
\end{example}
}

\ns{
The above example suggests that in order to get a unique result, we should fix a representation of quasi-copula $Q$.
A natural choice would be to represent $Q$ as supremum of all copulas that are below $Q$ in the point-wise order (as explained in Section~\ref{sec:mot} this is always possible in the bivariate setting).
However, even this does not work because the copula $M$ would be represented by the set of all copulas, and after shuffling with a shuffling mechanism $S$, the supremum would still be $M$, i.e., it would not coincide with the usual shuffle $M^S$.
For these reasons, the idea of indirect shuffling of quasi-copulas seem to be difficult to implement, unless a much more clever method is found.
Therefore, in this paper, we focus on direct shuffling of quasi-copulas instead.
}

It is well known (see \cite{Nels}) that any $d$-variate copula $C$ induces a $\sigma$-additive probability measure $\mu_C$ on the Borel $\sigma$-algebra on $\II^d$. On the other hand, there exist quasi-copulas that do not induce a signed measure on the same $\sigma$-algebra (see \cite{NeQuMoRoLaUbFl3,FeSaRoLaUbFl}).
 Nevertheless, any $d$-variate quasi-copula $Q$ induces a finitely-additive signed measure $\mu_Q$ on the algebra of finite disjoint unions of half-open boxes in $\II^d$, i.e., boxes of the from $R=(x_1,y_1]\times(x_2,y_2]\times\ldots\times(x_d,y_d]$, where $0 \leqslant x_i<y_i\leqslant 1$ for all $i \in [d]$.
    In fact, $\mu_Q$ is just an extension of the volume function $V_Q$ and is uniquely determined by the property $\mu_Q(R)=V_Q(R)$ for all half-open boxes $R \subseteq \II^d$. 
    Conversely, a finitely-additive signed measure $\mu$ on the algebra of finite disjoint unions of half-open boxes in $\II^d$ induces a grounded function $Q_\mu$ defined by
\begin{equation}\label{eq:gf}
Q_\mu(x_1,x_2,\ldots,x_d)=\mu\big((0,x_1]\times(0,x_2]\times\ldots\times(0,x_d]\big)
\end{equation}
for all $(x_1,x_2,\ldots,x_d) \in \II^d$.
To characterize when $Q_\mu$ is a quasi-copula we recall the definition of \emph{primal sets} from \cite{OmSt4} and adjust it to define \emph{half-open primal sets}. A \emph{half-open primal set} is any set of the form
\begin{align*}
R &=(0,y_1] \times\ldots\times(0,y_{k-1}]\times (x_k,y_k]\times (0,y_{k+1}]\times\ldots\times(0,y_d] \quad \textup{or}\\
R_c &=\big((0,1] \times\ldots\times(0,1]\times (x_k,y_k]\times (0,1]\times\ldots\times(0,1]\big)\backslash R
\end{align*}
where $k \in [d]$ and $0\leqslant x_k\leqslant y_k \leqslant 1$.
Observe that any half-open primal set is a finite union of half-open boxes. Hence, from \cite[Theorem~2]{OmSt4} we immediately obtain the following
\begin{proposition}
Function $Q_\mu$ induced by a finitely additive measure $\mu$ through Equation~\eqref{eq:gf} is a quasi-copula if and only if
\begin{enumerate}[$(i)$]
\item $\mu$ is stochastic, i.e., $\mu\big((0,1]\times\ldots\times(0,1]\times(x_i,y_i]\times(0,1]\times\ldots\times(0,1]\big)=y_i-x_i$ for all $i \in [d]$ and all $0\leqslant x_i<y_i\leqslant 1$, and
\item $\mu(T)$ is nonnegative for any half-open primal set $T$.
\end{enumerate}
\end{proposition}
We will now generalize the definition of a shuffle to the setting of quasi-copulas. Informally, a shuffle of a quasi-copula $Q$ can be described similarly as a shuffle of min (cf. Section~2 \ns{and the definition given in \cite{MiTa}}). However, for the formal definition \ns{we will use the interpretation with the interval exchange transformations (see \cite{DuFeSaTr}) and the push-forward of the} induced finitely-additive measure $\mu_Q$.

\begin{definition}
Let $Q$ be a $d$-variate quasi-copula and $\mu_Q$ its induced finitely-additive measure.
For each $k \in [d]$ let $\{J_i^k=(a_{i-1}^k,a_i^k]\}_{i=1}^{n_k}$ be a partition of $(0,1]$, where $0=a_0^k\leqslant a_1^k\leqslant \ldots \leqslant a_{n_k}^k=1$, and let $\pi_k$ be a permutation in $S_{n_k}$.
Furthermore, let $\phi \colon \II^d \to \II^d$ be a map given by $\phi(t_1,t_2,\ldots,t_d)=(\phi_1(t_1),\phi_2(t_2),\ldots,\phi_d(t_d))$, where $\phi_k \colon \II \to \II$ is a bijective piecewise linear measure preserving transformation defined by $\phi_k(0)=0$ and
$$\phi_k(t)=t-a_{i-1}^k+\sum_{\{j |\pi_k(j)<\pi_k(i)\}}\lambda(J_j^k)$$
for all $t \in J_i^k$ and $i \in [n_k]$, and $\lambda$ denotes the Lebesgue measure. If $\widehat{\mu}$ is the finitely-additive measure given by $\widehat{\mu}(R)=\mu_Q(\phi^{-1}(R))$ for all half-open boxes $R$, then the grounded function $Q_{\widehat{\mu}}$ induced by $\widehat{\mu}$ is called a \emph{shuffle} of $Q$.
\end{definition}

We remark that the finitely-additive measure $\widehat{\mu}$ above is well defined since the transformations $\phi_k$, $k \in [d]$, are piecewise linear, so that for a half-open box $R$ the set $\phi^{-1}(R)$ is a finite union of half-open boxes.
Note that our definition of a shuffle (of a quasi-copula) differs from the definition of a shuffle (of a copula) given in \cite{MiTa} by the fact that we shuffle also in the first coordinate while in \cite{MiTa} the first coordinate is not shuffled ($\phi_1(t)=t$ there).
We do this for convenience because we will encounter shuffles that use the same permutation for shuffling in all coordinates.
It should also be noted that a shuffle of a proper quasi-copula need not be a quasi-copula since the negative mass may be shuffled to any position.

Clearly, Theorem~\ref{thm:miku} cannot be directly generalized to quasi-copulas since the set of copulas is closed in the supremum norm.
However, the result of Theorem~\ref{thm:miku} can be interpreted in another way.
Suppose we are given a pair $(F,M)$, where $F$ is an arbitrary function and $M$ the upper Fr\'{e}chet-Hoeffding bound for copulas. Clearly, if we want to obtain the first factor of this pair by shuffling the second one, we must insist that $F$ is a shuffle of min.
However, if we allow approximation of the first factor, then we can let $F$ be an arbitrary copula. Indeed, function $F$ is a copula if and only if it can be approximated arbitrarily well in the supremum norm by a shuffle of $M$. Bearing in mind that the set of all copulas is closed, this is just a reformulation of Theorem~\ref{thm:miku}.
In order to be able to generalize Theorem~\ref{thm:miku} to quasi-copulas we thus have to allow approximation of the second factor as well.
So the question we pose is whether any quasi-copula can be approximated by a shuffle of some approximation of $M$. As we shall prove, the answer is positive for any $d \geq 2$.
Here we give a proof for the bivariate case only, since we will later prove a more general result valid in any dimension.
Nevertheless, we feel that it is worth writing down the proof for this special case first since the procedure we use is easier to visualize in this setting.

\begin{theorem}\label{thm:bishuffle}
Let $Q$ be a bivariate quasi-copula and $\varepsilon>0$. Then there exist bivariate quasi-copulas $A$ and $B$ such that $\|Q-A\|_\infty <\varepsilon$, $\|M-B\|_\infty<\varepsilon$, and $A$ is a shuffle of $B$.
\end{theorem}

\begin{proof}
  Fix a positive integer $m$ for now and let $n=m^2$.
Define $\{J_i\}_{i=1}^n$ as the regular partition of $\II$ into $n$ subintervals of equal length $\frac{1}{n}$ and $\{K_i\}_{i=1}^m$ as the regular partition of $\II$ into $m$ subintervals of equal length $\frac{1}{m}$, so that
$$K_i=\bigcup_{j=1}^m J_{m(i-1)+j}$$
for all $i \in [m]$.
In order to make the rest of the proof clearer we will call subsquares $J_i \times J_j$ of $\II^2$ the \emph{small squares} and squares $K_i \times K_j$ of $\II^2$ the \emph{big squares}.

Given a quasi-copula $Q$ we let $B$ be the ``$m$-th ordinal multiple of $Q$'' (see \cite{MeSe}), i.e., $$B(u,v)=a_i+(b_i-a_i) Q\left(\dfrac{u-a_i}{b_i-a_i}, \dfrac{v-a_i}{b_i-a_i} \right)$$ for $u,v\in K_i$, $i\in[m]$, and $B(u,v)=M(u,v)$ otherwise. Here, we have denoted $K_i=[a_i,b_i]$, so that $b_i-a_i=\frac{1}{m}$ for all $i\in[m]$.
Next, we slightly adjust the approach of \cite[Section~3.2.3]{Nels}. We introduce a (two-way) shuffle corresponding to the partition $\{J_i\}_{i=1}^n$ induced by the permutation $\pi_n$ on $n$ indices defined by $\pi_n(m(j-1)+k)=m(k-1)+j$ for all $j,k \in [m]$. We shuffle both rows and columns of $B$ using the same permutation $\pi_n$.

Let us compare the values of quasi-copula $Q$ and function $A$ obtained as the shuffle of $B$ described above. Denote by $\mu_{rs}$ the volume corresponding to $Q$ of the big square with northeast corner equal to $\left(\frac{r}{m}, \frac{s}{m}\right)$ for $r,s\in[m]$, i.e. $\mu_{rs}=V_Q(K_r \times K_s)$. Then
\begin{equation*}
    Q\left(\dfrac{r}{m}, \dfrac{s}{m}\right)=\sum_{j=1}^{r}\sum_{k=1}^{s}\mu_{jk}.
\end{equation*}
Now, all the nonzero mass of $B$ is concentrated in the $m$ big squares that lie along the main diagonal. Since each of these squares contains a normalized $Q$, the $B$-volume of each of them is $\frac{1}{m}$. We observe in passing that copula $M$ has all its mass concentrated on the diagonal so that the $M$-volume of every big square along the main diagonal is also $\frac{1}{m}$.
Each of the diagonal big squares, say the $l$-th one for $l\in[m]$, is a scaled image of $Q$. So, the volume of the small square at the $(r,s)$-th relative position within the $l$-th square equals $\dfrac{\mu_{rs}}{m}$ for every $r,s\in[m]$.
The global position of this small square is $(m(l-1)+r,m(l-1)+s)$ and is sent by the two-way shuffle onto the small square positioned at $(m(r-1)+l, m(s-1)+l)$. Therefore, when $l$ runs, the image of the small square moves along the small squares that lie on the diagonal of the big square at the $(r,s)$-th global position within $A$. Note that the small squares that do not lie on the diagonal of the big squares within $A$ all have the $A$-volume equal to $0$.
We conclude that the $A$-volume of the big square at position $(r,s)$ equals $\mu_{rs}$.
Furthermore, function $A$ is a quasi-copula, because when shuffling $B$ in each direction, the slices of the big box $K_i \times K_i$ containing a copy of $Q$ do not actually change relative position but only get spread out and only zero mass is inserted in between. Hence, after the shuffle any rectangle touching the boundary of $\II^2$ still has nonnegative volume.
Figure~\ref{fig:bishuffle} shows the mass distribution of an example quasi-copula $Q$ (left) and corresponding quasi-copulas $B$ (middle) and $A$ (right) for $m=4$.
  
\begin{figure}[!ht]
\includegraphics{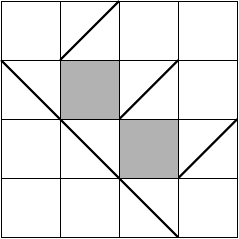}
\includegraphics{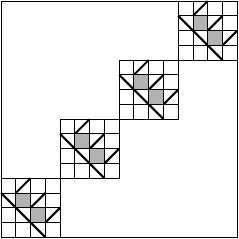}
\includegraphics{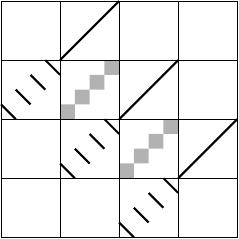}
\caption{Mass distribution of an example quasi-copulas $Q$ (left) and corresponding quasi-copulas $B$ (middle) and $A$ (right) from the proof of Theorem~\ref{thm:bishuffle} for $m=4$. On each big square quasi-copula $Q$ spreads either mass $\frac{1}{4}$ uniformly on one of the diagonals, or mass $-\frac{1}{4}$ uniformly on the square, or mass $0$.}
\label{fig:bishuffle}
\end{figure}

Using the analog of the displayed equation above for quasi-copula $A$ we deduce that $A\left(\frac{r}{m}, \frac{s}{m}\right)= Q\left(\frac{r}{m}, \frac{s}{m}\right)$ for all $r,s\in[m]$. This fact is clearly valid also for $r$ or $s$ equal to zero. Finally, choose arbitrary  $u,v\in\II$ and find $r,s\in[m]\cup\{0\}$ such that
  \[
    \left|u-\dfrac{r}{m}\right|<\dfrac{1}{m}\quad\mbox{and}\quad \left|v-\dfrac{s}{m}\right|<\dfrac{1}{m}.
  \]
  Then
  \[
    |A(u,v)-Q(u,v)|\leqslant\left|A(u,v)-A\left(\dfrac{r}{m}, \dfrac{s}{m}\right)\right| + \left|Q\left(\dfrac{r}{m}, \dfrac{s}{m}\right)-Q(u,v)\right| \leqslant \dfrac{2}{m}+ \dfrac{2}{m}
  \]
  by the fact that both quasi-copulas $A$ and $Q$ have the 1-Lipschitz property in each direction.
It remains to make an analog conclusion for $B$ and $M$. Indeed, recall the observation that $B$-volumes and $M$-volumes of big squares along the main diagonal are equal, implying that $B$-volumes and $M$-volumes of all big squares are equal,
inferring first that $B\left(\frac{r}{m}, \frac{s}{m}\right)= M\left(\frac{r}{m}, \frac{s}{m}\right)$ for all $r,s\in[m]\cup\{0\}$ and then, similarly as above,
$$|B(u,v)-M(u,v)| \leqslant \dfrac{2}{m}+ \dfrac{2}{m}.$$
The desired claim now follows if we choose the integer $m$ to be greater than $\frac{4}{\varepsilon}$.
\end{proof}

\begin{remark}
  It should be remarked that there are two major differences between copulas and proper quasi-copulas when shuffles are involved.
Firstly, a shuffle of a copula is always a copula, while a shuffle of a proper quasi-copula need not be a quasi-copula since a shuffle can easily bring negative volume to the boundary of $\II^d$.
So the situation described by the above theorem is, in fact, a very special situation where a shuffle of a quasi-copula is again a quasi-copula.
Secondly, since copulas induce positive measures on the Borel $\sigma$-algebra in $\II^d$ one can define more general shuffles of copulas than those considered above by using general measure preserving transformation, see for example \cite{DuFeSa}.
A quasi-copula $Q$, on the other hand, need not induce a signed measure on the Borel $\sigma$-algebra, but only induces a finitely-additive signed measure on the algebra of finite disjoint unions of half-open boxes. This means that for general quasi-copulas only basic shuffles considered above seem to be available.
\end{remark}

We raise the following open question. Does there exist a quasi-copula $Q$ with the property that the only shuffle of $Q$, which is again a quasi-copula, is $Q$ itself? Even more strict, does there exist a quasi-copula $Q$, which cannot be shuffled at all, i.e., the only shuffle of $Q$, which is again a quasi-copula, is a shuffle with identity permutations.

\section{Multivariate approximation}\label{sec:gen}

We will now generalize the result of Theorem~\ref{thm:bishuffle} by considering the general multivariate setting and by replacing the copula $M$ with an arbitrary quasi-copula. We first develop the necessary notation.

For a fixed positive integer $m$ let $n=m^2$, and define $\{J_i\}_{i=1}^n$ as the regular partition of $\II$ into $n$ subintervals of equal length $\frac{1}{n}$, and $\{K_i\}_{i=1}^m$ as the regular partition of $\II$ into $m$ subintervals of equal length $\frac{1}{m}$, so that
$$K_i=\bigcup_{j=1}^m J_{m(i-1)+j}$$
for all $i \in [m]$. These partitions create a division of the unit box $\II^d$ into $m^d$ boxes of width $\frac{1}{m}$ called the \emph{big boxes}, and each of those further into $m^d$ boxes of width $\frac{1}{n}$ called the \emph{small boxes}.

Let $\mathbf{r}=(r_1,r_2,\ldots,r_d)$ be a $d$-tuple of indices $r_i\in[m]$ for $i\in[d]$ and denote by ${\boldsymbol\xi}_{\mathbf{r}}$ the big box with the corner of maximal values equal to $\left(\dfrac{r_1}{m},\dfrac{r_2}{m},\ldots, \dfrac{r_d}{m}\right)$. It makes sense to say that $d$-tuple $\mathbf{r}$ \emph{determines the position} of this box. When $\mathbf{r}$ runs through all $d$-tuples in $[m]^d$, ${\boldsymbol\xi}_{\mathbf{r}}$ runs through all possible big boxes in $\II^d$.
Now, each small box within a fixed big box ${\boldsymbol\xi}_{\mathbf{r}}$ is determined in a similar way by a $d$-tuple $\mathbf{s}=(s_1,s_2,\ldots,s_d)$ of indices $s_i\in[m]$ for $i\in[d]$ giving the relative position of the small box within the big one.
We will denote the small box with relative position $\mathbf{s}$ within the big box ${\boldsymbol\xi}_{\mathbf{r}}$ by ${\boldsymbol\xi}_{\mathbf{r}}^{\mathbf{s}}$. Clearly, the global position of this small box within $\II^d$ is determined by its corner of maximal values which is equal to
  \[
    \left(\dfrac{r_1-1}{m}+\dfrac{s_1}{n},\dfrac{r_2-1}{m}+\dfrac{s_2}{n},\ldots, \dfrac{r_d-1}{m}+\dfrac{s_d}{n}\right).
  \]
  Recall also the permutation $\pi_n$ on $n$ indices defined in the proof of Theorem~\ref{thm:bishuffle} by $\pi_n(m(j-1)+k)=m(k-1)+j$. We will again use this permutation to shuffle the elements of partition $\{J_i\}_{i=1}^n$ in all $d$ variates simultaneously.
  If we do so, we may interpret $\pi_n$ as a bijective mapping $\pi_n \colon \II^d \to \II^d$ that shuffles the small boxes in $\II^d$ but otherwise preserves the shuffled pieces, i.e., it translates each small box. In this case
\begin{equation}\label{eq:pi}
\pi_n({\boldsymbol\xi}_{\mathbf{r}}^{\mathbf{s}})= {\boldsymbol\xi}_{\mathbf{s}}^{\mathbf{r}} \quad \textup{for all }  \mathbf{r},\mathbf{s} \in [m]^d.
\end{equation}
If $Q$ is a quasi-copula and $\mu_Q$ its finitely-additive measure, then the finitely-additive measure given by $\mu_{\pi_n(Q)}(R)=\mu_Q(\pi_n^{-1}(R))$ for all half-open boxes $R$ is well defined since $\pi_n^{-1}(R)$ is a finite union of half-open boxes. We will denote the grounded function it induces by $\pi_n(Q)$. \MO{(Observe that this is a slight change of the notation from Section \ref{sec:shuff} due to simplification in further computations. This should cause no confusion since the partition of the shuffles is fixed here.)}

Now choose $A,B\in\mathcal{Q}$ and let $\mu_A$ and $\mu_B$ denote their respective finitely-additive measures.
For convenience we will henceforth assume that the boxes ${\boldsymbol\xi}_{\mathbf{r}}^{\mathbf{s}}$ are half-open.
Let us introduce
a \emph{tensor product} of $A$ and $B$, depending on integer $m$, as a grounded function $A \otimes_m B$ induced by the finitely-additive measure $\mu_{A\otimes_m B}$ determined by
\begin{equation}\label{eq:tensor1}
\mu_{A\otimes_m B}(R) =\sum_{\mathbf{r},\mathbf{s} \in [m]^d}{\mu_A\left({{\boldsymbol\xi}_{\mathbf{r}}}\right) \mu_B\left( {{\boldsymbol\xi}_\mathbf{s}}\right)}  \frac{\mu_\Pi(R \cap {{\boldsymbol\xi}_{\mathbf{r}}^\mathbf{s}})}{\mu_\Pi({{\boldsymbol\xi}_{\mathbf{r}}^\mathbf{s}})},
\end{equation}
where $\Pi$ denotes the product copula and $R$ is an arbitrary half-open box in $\II^d$. Hence, for all $\mathbf{r},\mathbf{s} \in [m]^d$ the tensor product $A \otimes_m B$ spreads mass $\mu_A\left({{\boldsymbol\xi}_{\mathbf{r}}}\right) \mu_B\left( {{\boldsymbol\xi}_\mathbf{s}}\right)$ uniformly on ${{\boldsymbol\xi}_{\mathbf{r}}^\mathbf{s}}$. In particular,
\begin{equation}\label{eq:tensor2}
\mu_{A\otimes_m B}\left({{\boldsymbol\xi}_{\mathbf{r}}^\mathbf{s}}\right) ={\mu_A\left({{\boldsymbol\xi}_{\mathbf{r}}}\right) \mu_B\left( {{\boldsymbol\xi}_\mathbf{s}}\right)}  \quad\mbox{for all } \mathbf{r},\mathbf{s}\in[m]^d.
\end{equation}
    
We remark that when $d=2$ this construction is related to the Kronecker product (or tensor product) of matrices in the following way.
Let $\Delta_m$ be the mesh of all vertices of the big boxes ${\boldsymbol\xi}_{\mathbf{r}}$ and let $A_m$ and $B_m$ be discrete quasi-copulas obtained by restricting $A$ and $B$ to $\Delta_m$.
Furthermore, let $M_{A_m}$ and $M_{B_m}$ denote the mass matrices of the discrete quasi-copulas $A_m$ and $B_m$, i.e., the entries of $M_{A_m}$ are the $A$-volumes of the big boxes ${\boldsymbol\xi}_{\mathbf{r}}$, $\mathbf{r} \in [m]^2$.
Now let $\Delta_n$ be the mesh of all vertices of the small boxes ${{\boldsymbol\xi}_{\mathbf{r}}^\mathbf{s}}$, $\mathbf{r},\mathbf{s} \in [m]^2$, and let $T$ be a grounded discrete function defined on $\Delta_n$ whose mass matrix is the Kronecker product $M_{A_m} \otimes M_{B_m}$.
Then $T$ is the restriction of $A\otimes_m B$ to $\Delta_n$ and $A\otimes_m B$ is a piecewise bilinear extension of $T$ to $\II^2$.
So in terms of discrete quasi-copulas, the tensor product is given by the Kronecker product of the corresponding mass matrices.

Constructions similar to the tensor product defined above already exist in the literature.
In particular, in \cite{FeSaRoLaUbFl} the authors consider transformations of bivariate quasi-copulas induced by the so called \emph{quasi-transformation matrices}.
An interested reader will notice that a \emph{quasi-transformation matrix} is just the mass matrix of a bivariate discrete quasi-copula defined on a finite (not necessarily equidistant) mesh of points in $\II^2$.
Given a quasi-copula $Q$ and a quasi-transformation matrix $M=(m_{ij})$ (where $i$ increases left to right and $j$ increases bottom to top), whose associated mesh splits $\II^2$ into rectangles $R_{ij}$, the authors introduce a quasi-copula $M(Q)$ that, for each $i$ and $j$, spreads mass $m_{ij}$ in rectangle $R_{ij}$ in the same but rescaled way as $Q$ spreads mass in $\II^2$.
Using the notation from the previous paragraph and the transformation just described, the tensor product of bivariate quasi-copulas $A$ and $B$ can be expressed as $A\otimes_m B=M_{A_m}(M_{B_m}(\Pi))$, where $\Pi$ is the product copula.
Note that the reason why copula $\Pi$ appears in this formula is that in the definition of the tensor product we chose to spread mass uniformly on each small box ${{\boldsymbol\xi}_{\mathbf{r}}^\mathbf{s}}$.
We could easily have chosen to spread mass using some other copula or quasi-copula as the pattern. The crucial point here is that in all small boxes the mass is spread equally, i.e., using the same fixed copula.

The extension of the above construction $M(Q)$ to multivariate quasi-copulas was introduced only recently in \cite{FeSaQuMoUbFl}.
It has previously been considered for multivariate copulas $C$ in \cite{TrFeSa,FeSaTr}, where the matrix $M$ is replaced by a probability distribution $\tau$ on a discrete mesh of points (called a \emph{generalized transformation matrix}) and the resulting copula is denoted by $\mathcal{V}_\tau(C)$.
Here, we will formulate the extension in terms of discrete functions rather than quasi-transformation matrices since this approach seems to be applicable to all situations, bivariate and multivariate, for copulas and quasi-copulas.

\MO{In what follows we relax our definition of a $d$-variate discrete quasi-copula by allowing boxes in a mesh to be of different sizes.}
Let $A$ be a $d$-variate discrete quasi-copula defined on a mesh determined by half-open boxes $R_\mathbf{i}=(a_{i_1}^1,b_{i_1}^1] \times (a_{i_2}^2,b_{i_2}^2]\times\ldots\times (a_{i_d}^d,b_{i_d}^d] $ for $\mathbf{i}\in \mathcal{I}= [n_1]\times[n_2]\times\ldots\times [n_d]$,
and let $M_A=(m_\mathbf{i})_{\mathbf{i} \in \mathcal{I}}$ be its $d$-dimensional mass table, i.e., $m_\mathbf{i}=V_A(R_\mathbf{i})$ for all $\mathbf{i}\in \mathcal{I}$.
\MO{For any $\mathbf{i}\in \mathcal{I}$ and any $\mathbf{b}_\mathbf{i}=(b_{i_1}^1,b_{i_2}^2,\ldots,b_{i_d}^d)$ we have
$$A (\mathbf{b}_\mathbf{i})=\sum_{\mathbf{b}_\mathbf{j}\leqslant\mathbf{b}_\mathbf{i}} m_\mathbf{j},$$ 
where $\mathbf{b}_\mathbf{j}\leqslant\mathbf{b}_\mathbf{i}$ means componentwise order. If any of the components of $\mathbf{b}_\mathbf{i}$ is zero, the value of $A$ at this point is left undefined. By abuse of notation we will understand that this value is zero. It is not hard to see that the so defined function $A$, defined on the set $\{\mathbf{b}_\mathbf{i}\}_{\mathbf{i}\in \mathcal{I}}\subseteq\II^d$, satisfies the definition of (discrete) quasi-copula from the beginning of Section \ref{sec:mot}. 
}

For each $\mathbf{i} \in \mathcal{I}$ let $\phi_\mathbf{i} \colon R_\mathbf{i} \to (0,1]^d$ be a bijective linear rescaling map defined by
$$\phi_\mathbf{i}(t_1,t_2,\ldots,t_d)=\left(\frac{t_1-a_{i_1}^1}{b_{i_1}^1-a_{i_1}^1},\frac{t_2-a_{i_2}^2}{b_{i_2}^1-a_{i_2}^2},\ldots,\frac{t_d-a_{i_d}^d}{b_{i_d}^d-a_{i_d}^d}\right).$$
For a quasi-copula $Q$ let $M_A(Q)$ denote the grounded function induced by the finitely-additive measure $\mu_{M_A(Q)}$ determined by the property
$$\mu_{M_A(Q)}(R)=m_\mathbf{i}\cdot \mu_{Q}(\phi_\mathbf{i}(R))$$
whenever $R$ is a half-open subbox of $R_\mathbf{i}$. Loosely speaking, function $M_A(Q)$ spreads mass on each $R_\mathbf{i}$ in the same but rescaled way as $Q$ spreads mass on $\II^d$.
The following lemma gives the discussed extension and is an easily consequence of \cite[Theorem~10]{FeSaQuMoUbFl} and the discussions in \cite{DeBaDeMe} and \cite{KoBuKoOmSt}.

\begin{lemma}\label{lem:transformation}
For any discrete quasi-copula $A$ and any quasi-copula $Q$, the function $M_A(Q)$ is a quasi-copula.
\end{lemma}

Note that with this extension the tensor product formula
$$A\otimes_m B=M_{A_m}(M_{B_m}(\Pi))$$
remains valid in any dimension, so Lemma~\ref{lem:transformation} implies that $A \otimes_m B$ is a quasi-copula for any quasi-copulas $A$ and $B$.
The following proposition is the main reason for our introduction of a tensor product of quasi-copulas.

  \begin{proposition}\label{prop:tensor}
For any two quasi-copulas $A$ and $B$ the following properties hold.
  \begin{enumerate}[$(a)$]
  \item $A \otimes_m B$ and $B \otimes_m A$ are quasi-copulas.
    \item Shuffle \MO{determined by permutation} $\pi_n$ defined in \eqref{eq:pi} has the property $\pi_n(A\otimes_m B)=B\otimes_m A$, so that quasi-copulas $A\otimes_m B$ and $B\otimes_m A$ are shuffles of each other.
    \item Quasi-copula $A\otimes_m B$ 
    converges uniformly to $A$ 
    as $m$ tends to infinity.
  \end{enumerate}
  \end{proposition}

  \begin{proof}
Property $(a)$ follows from the above discussion. Equations~\eqref{eq:pi} and \eqref{eq:tensor2} imply that
$$\mu_{\pi_n(A\otimes_m B)}({\boldsymbol\xi}_\mathbf{r}^\mathbf{s})=\mu_{A\otimes_m B}(\pi_n^{-1}({\boldsymbol\xi}_\mathbf{r}^\mathbf{s}))=\mu_{A\otimes_m B}({\boldsymbol\xi}_\mathbf{s}^\mathbf{r})=\mu_A({\boldsymbol\xi}_\mathbf{s})\mu_B({\boldsymbol\xi}_\mathbf{r})=\mu_{B\otimes_m A}({\boldsymbol\xi}_\mathbf{r}^\mathbf{s})$$
for all $\mathbf{r},\mathbf{s} \in [m]^d$. Furthermore,  Equations~\eqref{eq:tensor1} and \eqref{eq:pi} imply that both $B\otimes_m A$ and $\pi_n(A\otimes_m B)$ spread mass uniformly on ${\boldsymbol\xi}_\mathbf{r}^\mathbf{s}$. Hence, $(b)$ easily follows.
      To see $(c)$ let us compute the $(A\otimes_m B)$-volume of the big boxes ${\boldsymbol\xi}_\mathbf{r}$, $\mathbf{r} \in [m]^d$. Using Equation~\eqref{eq:tensor2} we get
    \[
        \mu_{A\otimes_m B} \left({{\boldsymbol\xi}_{\mathbf{r}}} \right)= \sum_{\mathbf{s} \in [m]^d}\mu_{A\otimes_m B} \left({{\boldsymbol\xi}_{\mathbf{r}}^{\mathbf{s}}} \right)= \sum_{\mathbf{s} \in [m]^d}{\mu_A\left({{\boldsymbol\xi}_{\mathbf{r}}}\right) \mu_B\left( {{\boldsymbol\xi}_\mathbf{s}}\right)} = \mu_A\left({{\boldsymbol\xi}_{\mathbf{r}}}\right)
    \]
    since the sum of $\mu_B({\boldsymbol\xi}_\mathbf{s})$ over all $\mathbf{s}\in [m]^d$ equals $1$. This implies
  \[
    (A\otimes_m B)\left(\dfrac{r_1}{m},\dfrac{r_2}{m},\ldots, \dfrac{r_d}{m}\right)= \sum_{\mathbf{s} \leqslant \mathbf{r}} \mu_{A\otimes_m B} \left({\boldsymbol\xi}_{\mathbf{s}}\right) =A\left(\dfrac{r_1}{m},\dfrac{r_2}{m},\ldots, \dfrac{r_d}{m}\right),
  \]
  where $\mathbf{r}=(r_1,r_2,\ldots,r_d)$ and $\mathbf{s} \leqslant \mathbf{r}$ if and only if $s_i \leqslant r_i$ for all $i \in [d]$.
So the values of $A\otimes_m B$ and $A$ are equal at all points $\frac{1}{m}{\boldsymbol r}$ for all $\mathbf{r}\in[m]^d$. This fact is clearly valid also if any entry of $\mathbf{r}$ equals zero.
Now let $\mathbf{u}=(u_1,u_2,\ldots,u_d) \in \II^d$ be arbitrary and find $\mathbf{r}\in([m]\cup\{0\})^d$ such that $\left| u_i-\frac{r_i}{m}\right|<\frac{1}{m}$ for all $i\in[d]$. Then
  \[\begin{split}
       |(A\otimes_m B)(\mathbf{u})-A(\mathbf{u})| & \leqslant\left|(A\otimes_m B)(\mathbf{u})-(A\otimes_m B)(\tfrac{1}{m}{\boldsymbol r})\right|  + \left|A(\tfrac{1}{m}{\boldsymbol r})-A(\mathbf{u})\right| \\
       & \leqslant \dfrac{d}{m}+ \dfrac{d}{m}
    \end{split}
  \]
  by the fact that both quasi-copulas $A\otimes_m B$ and $A$ have the 1-Lipschitz property in each direction. This finishes the proof of $(c)$
  .
  \end{proof}

As a direct consequence of Proposition~\ref{prop:tensor} we can now easily generalize Theorem~\ref{thm:bishuffle}.

\begin{theorem}\label{thm:multishuffle}
Let $Q_1$ and $Q_2$ be $d$-variate quasi-copulas and $\varepsilon>0$. Then there exist $d$-variate quasi-copulas $A_1$ and $A_2$ such that $\|Q_1-A_1\|_\infty <\varepsilon$, $\|Q_2-A_2\|_\infty<\varepsilon$, and $A_1$ is a shuffle of $A_2$.
\end{theorem}

In particular, taking $Q_2=M$ we get the announced generalization of Theorem~\ref{thm:miku}.

\begin{theorem}
Let $\mathcal{Q}_d$ be the set of all $d$-variate quasi-copulas equipped with the uniform distance and $\mathcal{S}$ the set of shuffles of an arbitrarily small neighborhood of copula $M \in \mathcal{Q}_d$. Then $\mathcal{S} \cap \mathcal{Q}_d$ is a dense subset of $\mathcal{Q}_d$.
\end{theorem}

\ns{
Note that 
Theorem~\ref{thm:bishuffle} is a special case of the bivariate Theorem~\ref{thm:multishuffle}. This case is interesting because $M$ is a copula as opposed to being a general quasi-copula.
Furthermore, bivariate copula $M$ has a stochastic interpretation of being the dependence structure that describes two comonotonic continuous random variables, i.e., random variables that are almost surely increasing functions of each other, see \cite[\S 2.5]{DuSe}.
The quasi-copula $B$ from the proof of Theorem~\ref{thm:bishuffle} is an ordinal sum of quasi-copulas so it can be represented as the supremum of ordinal sums of copulas.
In this case, the particular type of shuffles used in the proof of Theorem~\ref{thm:bishuffle} can be interpreted as indirect shuffles, as the following theorem explains. Here, we use $\leqslant$ and $\geqslant$ to denote the pointwise order on quasi-copulas, i.e., $A \leqslant B$ if and only if $A(u,v) \leqslant B(u,v)$ for all $u,v \in \II$.}

Note that the following theorem holds for bivariate quasi-copulas, as pointed above.

{\begin{theorem}\label{thm:stochastic}
Let $Q$ be an ordinal sum of quasi-copulas $\{Q_i\}_{i=1}^m$ with respect to the partition $\{K_i\}_{i=1}^m$ of $\II$. Furthermore, let $n=n_1+n_2+\ldots+n_m$ and let $\{J_i\}_{i=1}^n$ be a refinement of the partition $\{K_i\}_{i=1}^m$ such that for every $i \in [m]$ we have
$$K_i=\bigcup_{j=1}^{n_i} J_{\overline{n}_i+j}, \qquad \text{where } \overline{n}_i=n_1+n_2+\ldots+n_{i-1} .$$
Fix a permutation $\pi$ of the intervals $\{J_i\}_{i=1}^n$ that preserves the relative position of the intervals $\{J_{\overline{n}_i+j}\}_{j=1}^{n_i}$ for every $i \in [m]$ and denote the corresponding shuffling mechanism by $S$. For a quasi-copula $A$ denote by $A^{(S,S)}$ the shuffle of $A$, where we shuffle in both directions using $S$.
Then
$$Q^{(S,S)}(u,v)=\inf\{C^{(S,S)}(u,v) ~|~ C \geqslant Q\}$$
for all $u,v \in \II$.
\end{theorem}}

{\begin{proof}
Note that for a copula $C$ we have $C \geqslant Q$ if and only if $C$ is an ordinal sum of some copulas $\{C_i\}_{i=1}^m$ with respect to the partition $\{K_i\}_{i=1}^m$ such that $C_i \geqslant Q_i$ for all $i \in [m]$.
This is because $Q$ equals $M$ outside the union of squares $K_i \times K_i$, $i=1,2,\ldots,m$, hence, so does any copula $C$ with $C \geqslant Q$ and the rest is obvious.
Now fix $u,v \in \II$. Since $S$ preserves the relative positions of the intervals $\{J_{\overline{n}_i+j}\}_{j=1}^{n_i}$ for every $i \in [m]$, there exist $u_i,v_i \in \II$ such that
$$Q^{(S,S)}(u,v)=\frac{Q_1(u_1,v_1)}{\lambda(K_1)}+\frac{Q_2(u_2,v_2)}{\lambda(K_2)}+\ldots+\frac{Q_m(u_m,v_m)}{\lambda(K_m)},$$
where $\lambda(K_i)$ denotes the length of $K_i$.
Although an explicit formula for $u_i$ and $v_i$ could be given it is perhaps best to just illustrate them on a figure, see Figure~\ref{fig:relative}.
\begin{figure}[!ht]\ns{
\centering
\includegraphics{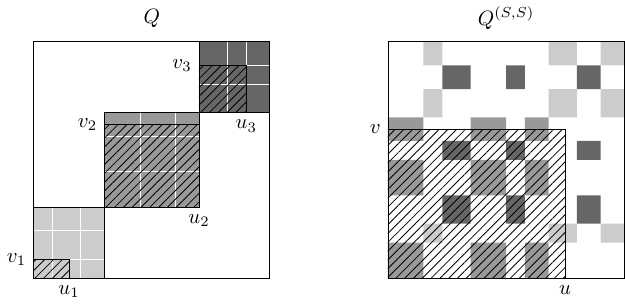}
\caption{The mass distribution of a quasi-copula $Q$ that is an ordinal sum of $3$ quasi-copulas (left) and the mass distribution of the corresponding two-way shuffle $Q^{(S,S)}$ (right) induced by the permutation
$\pi=\begin{pmatrix}
1&2&3&4&5&6&7&8&9\\
4&1&7&5&8&6&2&9&3
\end{pmatrix}$. The total mass in the stripped region on the right equals $Q^{(S,S)}(u,v)$ and the relative positions (within the respective squares) of the corresponding points $u_i$ and $v_i$ from the proof of Theorem~\ref{thm:stochastic} are depicted on the left.}
\label{fig:relative}}
\end{figure}
Similarly,
$$C^{(S,S)}(u,v)=\frac{C_1(u_1,v_1)}{\lambda(K_1)}+\frac{C_2(u_2,v_2)}{\lambda(K_2)}+\ldots+\frac{C_m(u_m,v_m)}{\lambda(K_m)},$$
for any $C \geqslant Q$.
This implies
\begin{align*}
\inf\{C^{(S,S)}(u,v) ~|~ C \geqslant Q\} &=
\inf\left\{\sum_{i=1}^m \frac{C_i(u_i,v_i)}{\lambda(K_i)} ~|~ C \geqslant Q\right\}
=\sum_{i=1}^m \inf\left\{\frac{C_i(u_i,v_i)}{\lambda(K_i)} ~|~ C_i \geqslant Q_i\right\} \\
&=\sum_{i=1}^m \frac{Q_i(u_i,v_i)}{\lambda(K_i)}=
Q^{(S,S)}(u,v),
\end{align*}
which completes the proof.
\end{proof}}

\ns{
Theorem~\ref{thm:stochastic} offers an indirect stochastic interpretation of the type of shuffles used in the proof of Theorem~\ref{thm:bishuffle}.
In particular, if an ordinal sum $Q$ of bivariate quasi-copula is interpreted as the lower envelope of the stochastic information contained in the set of copulas $\{C \text{ copula} ~|~ C \geqslant Q\}$, then $Q^{(S,S)}$ is the lower envelope of the shuffled information.
}

\section{Quasi-copula obeying a given signed mass pattern on a patch}\label{sec:patch+-}

\MO{This section is devoted to some special aspects of constructing quasi-copulas with respect to negative volumes. While these volumes are a nuisance for statistical interpretation, we get an amazing freedom in constructions, as discovered by Nelsen et al. in \cite[Theorem 4.2]{NeQuMoRoLaUbFl1}: Given any $\varepsilon>0$, no matter how small, and any $M > 0$, no matter how big, and any quasi-copula $Q$, there exists a quasi-copula $\widetilde{Q}$ and a set $S\subseteq\II^2$ such that: \emph{(a)} $\mu_{\widetilde{Q}}(S) < -M$, and \emph{(b)} $| Q(x,y) - \widetilde{Q}(x,y)| < \varepsilon$ for all $x,y \in \II$.}

\MO{However, our interest here does not lie so much in the magnitude of the negative volumes but more in how they are spread among the positive ones. It originates in a problem proposed by Montes et al. \cite{MoMiPeVi}: They call a pair of bivariate quasi-copulas $A\leqslant B$ an \emph{imprecise copula}, if they satisfy
\begin{enumerate}[$(a)$]
  \item $A(\mathbf{a}) + B(\mathbf{c}) - A(\mathbf{b}) - A(\mathbf{d}) > 0;$
  \item $B(\mathbf{a}) + A(\mathbf{c}) - A(\mathbf{b}) - A(\mathbf{d}) > 0;$
  \item $B(\mathbf{a}) + B(\mathbf{c}) - B(\mathbf{b}) - A(\mathbf{d}) > 0;$
  \item $B(\mathbf{a}) + B(\mathbf{c}) - A(\mathbf{b}) - B(\mathbf{d}) > 0;$
\end{enumerate}
for each rectangle defined by corners $\mathbf{a}, \mathbf{b}, \mathbf{c},$ and $\mathbf{d}$ in the standard way, i.e., starting from the south-west corner and continuing in positive orientation. Their question was whether for any imprecise copula in their sense there existed a true copula $C$ such that $A\leqslant C\leqslant B$. Their question was listed as the Hitchhiker's problem 6, and solved in the negative by the authors of this paper, our solution appearing in the same issue of the same journal \cite{OmSt1} as the 
list itself \cite{ArGaMeDeBa}. (This answer was making the definition of an imprecise copula proposed in \cite{MoMiPeVi} questionable.)
}

\MO{
The method characterising pairs of quasi-copulas $A\leqslant B$ such that a true copula $C$ exists with $A\leqslant C\leqslant B$, has been further developed and extended to the multivariate case in \cite{OmSt2,OmSt4}, and then called ALGEN method and used to solve also Hitchhiker's problem 3 in \cite{OmSt5}. Now, looking back from a five-year distance to our paper \cite{OmSt1}, we can see another ingredient in it that deserves to be developed further. The counterexample presented there was made of two quasi-copulas $A\leqslant B$ satisfying the above conditions of an imprecise copula such that there was no (!) copula $C$ with the property $A\leqslant C\leqslant B$. What was used there to create these quasi-copulas was an algorithm working tacitly behind the scenes, that 
distributes mass appropriately around a given patch; in other words, it creates quasi-copulas with a given pattern of plus and minus masses on a patch. Thus, we present here how this freedom in creating bivariate quasi-copulas works in general. One should be able to develop it further and extend it to the multivariate case in the future work.
}\\[1mm]

 Assume that on a certain rectangular patch $\PP$, divided into squares, masses are given obeying a fixed pattern of signs ``$+$'' and ``$-$''.
We want to extend the signed mass pattern to a square region containing $\PP$ in such a way that after a linear rescaling the obtained pattern will represent a mass distribution of some discrete quasi-copula.
For instance, if the pattern is given by masses
$\begin{pmatrix}
\mu & -\mu \\
-\mu & \mu
\end{pmatrix}$,
the total mass, i.e., then the volume of the patch is zero.
Besides the boundary condition, the most important conditions to fulfill are monotonicity and 1-Lipschitz condition in each direction.
These two conditions, however, can be replaced in the bivariate case by one equivalent geometric condition, expressed in terms of volumes. We require that all volumes of rectangles touching the boundary be nonnegative.
This means that in our example we need to add at least (in the language of matrices)
(1) mass $\mu$ south of the first column,
(2) mass $\mu$ north of the second column,
(3) mass $\mu$ east of the first row, and
(4) mass $\mu$ west of the second row, before anything else.

Assume that the patch $\PP$ is given by a rectangle $[0,m]\times[0,n]$ divided into small squares by lines $x=i$ and $y=j$ for $i\in[m]$ and $j\in[n]$, denoted by $R_{ij}=[i-1,i]\times[j-1,j]$. This notation of squares will be assumed theoretically on the whole plane, not just the patch.
The signed pattern is now defined by fixing arbitrary real values for the volumes $V(R_{ij})$. Observe that we might allow the total mass of the patch, i.e., the sum of all volumes, to be either negative, zero or positive. In order that the problem mentioned above is interesting, at least one of these volumes should be strictly negative. However, our investigation seems new and may have applications in the other cases as well. Observe also that neither the \emph{starting total mass of the patch}
\[
   V_0=\sum_{i=1}^{m}\sum_{j=1}^{n}V(R_{ij})
\]
nor the final total mass of the extended region, nor the final region itself are assumed bounded since we will rescale all of them later.

Here is our algorithm. For any $j\in[n]$ define
\[
    \mu^{h-}_j=-\min_{s\in[m]}\sum_{i=1}^{s}V(R_{ij})
\]
whenever the right-hand side is positive and $\mu^{h-}_j=0$ otherwise. Similarly, we let for any $j\in[n]$
\[
    \mu^{h+}_j=-\min_{s\in[m]}\sum_{i=s}^{m}V(R_{ij})
\]
whenever the right-hand side is positive and $\mu^{h+}_j=0$ otherwise.
These two values mean the respective maximal absolute value of a negative volume of a rectangle touching the left and the right boundary in the $j$-th row of the patch. When we add volume $\mu^{h-}_j$ on the left-hand side of the $j$-th row of the patch and volume $\mu^{h+}_j$ on the right-hand side of it, the total volume of the $j$-th row becomes
\[
    \mu^{h0}_j=\mu^{h-}_j + \sum_{i=1}^{m}V(R_{ij}) +\mu^{h+}_j.
\]
Of course, we also need for all $i\in[m]$ the corresponding values for all columns
\[
    \mu^{v-}_i=-\min_{s\in[n]}\sum_{j=1}^{s}V(R_{ij})\quad\mbox{and}\quad \mu^{v+}_i=-\min_{s\in[n]}\sum_{j=s}^{n}V(R_{ij})
\]
whenever the right-hand sides are positive and $0$ otherwise, as well as the corresponding total sums
\[
    \mu^{v0}_i=\mu^{v-}_i + \sum_{j=1}^{n}V(R_{ij}) +\mu^{v+}_i\quad\mbox{for all}\quad i\in[m].
\]
Finally, introduce the maximal row and column total sum
\[
    \mu=\max\left\{\max_{j\in[n]}\mu^{h0}_j,\max_{i\in[m]}\mu^{v0}_i\right\}
\]
Since our goal is a doubly stochastic signed mass pattern, this will be achieved by making the volumes of all rows and columns equal to $\mu$ and then normalizing. Observe that the so defined value $\mu$ is the smallest possible choice of the kind of value.

\begin{example}\label{ex:mass}
Let the signed mass pattern be given by
$$\left[\begin{array}{ccc}
0 & 0 & 3 \\
-2 & 1 & 0 \\
1 & -3 & 1
\end{array}\right],$$
where $i$ runs left-to-right and $j$ runs bottom-to-top.
Then $V_0=1$,
\begin{align*}
\mu^{h-}_1 &=2, &\mu^{h+}_1 &=2, &\mu^{v-}_1 &=1, &\mu^{v+}_1 &=2, &\mu^{h0}_1 &=3, &\mu^{v0}_1 &=2, \\
\mu^{h-}_2 &=2, &\mu^{h+}_2 &=1, &\mu^{v-}_2 &=3, &\mu^{v+}_2 &=2, &\mu^{h0}_2 &=2, &\mu^{v0}_2 &=3, \\
\mu^{h-}_3 &=0, &\mu^{h+}_3 &=0, &\mu^{v-}_3 &=0, &\mu^{v+}_3 &=0, &\mu^{h0}_3 &=3, &\mu^{v0}_3 &=4, 
\end{align*}
and $\mu=4$. 
\end{example}

The first step of the algorithm is to fulfill the \emph{Geometric condition} described in the first paragraph of this section. To this end we just insert positive masses in all four directions around the patch.
\begin{enumerate}[$(1)$]
\item \emph{Negative horizontal:} $V(R_{0,j})=\mu^{h-}_j$ for $j\in[n]$,
\item \emph{Positive horizontal:} $V(R_{n+1,j})=\mu^{h+}_j$ for $j\in[n]$,
\item \emph{Negative vertical:} $V(R_{i,0}) =\mu^{v-}_i$ for $i\in[m]$,
\item \emph{Positive vertical:} $V(R_{i,m+1})=\mu^{v+}_i$ for $i\in[m]$.
\end{enumerate}

\begin{example}\label{ex:geo}
For the mass pattern in Example~\ref{ex:mass} fulfilling the Geometric condition results in
$$\left[\begin{array}{ccc}
0 & 0 & 3 \\
-2 & 1 & 0 \\
1 & -3 & 1
\end{array}\right]
\longrightarrow
\left[\begin{array}{ccccc}
 & 2 & 2 & 0 &  \\
0 & \cc 0 & \cc 0 & \cc 3 & 0 \\
2 & \cc -2 & \cc 1 & \cc 0 & 1 \\
2 & \cc 1 & \cc -3 & \cc 1 & 2 \\
 & 1 & 3 & 0 & 
\end{array}\right].$$
\end{example}

The next step of the algorithm might be called \emph{Equalize and spread}. We want to achieve the doubly stochastic condition.
The idea is to set the volumes of all rows and columns of our enlarged patch to the same value and then normalize the volumes by dividing all the masses by the total mass. The minimal possible value that can be given this role is the above computed $\mu$.
In the rows and columns (going through the patch), that do not yet have total mass $\mu$, we may add the missing mass to the necessary positive masses on either side of each column and row.
There is a number of possible choices, say, we may add half of the missing mass on each side
\begin{align*}
&\mu^{h-}_j \rightsquigarrow x_j^-:=\mu^{h-}_j +\dfrac{\mu-\mu^{h0}_j}{2},
&\mu^{h+}_j &\rightsquigarrow x_j^+:=\mu^{h+}_j +\dfrac{\mu-\mu^{h0}_j}{2}, && \text{for all $j\in[n]$,}\\
&\mu^{v-}_i \rightsquigarrow y_i^-:=\mu^{v-}_i +\dfrac{\mu-\mu^{v0}_i}{2},
&\mu^{v+}_i &\rightsquigarrow y_i^+:=\mu^{v+}_i +\dfrac{\mu-\mu^{v0}_i}{2}, && \text{for all $i\in[m]$.}
\end{align*}
All these changes result in $\mu^{h0}_j \rightsquigarrow\mu$ and $\mu^{v0}_i \rightsquigarrow\mu$ for all $i\in[m]$ and $j\in[n]$.
The final step of the algorithm will depend on the choice of how we fork the missing mass among the positive and negative side of rows and columns. This part of the equalizing phase we call \emph{forking}.

\begin{example}\label{ex:fork}
For the mass patter from Example~\ref{ex:geo}, if we fork all additional mass to the west and north, we obtain
$$\left[\begin{array}{ccccc}
 & 2 & 2 & 0 &  \\
0 & \cc 0 & \cc 0 & \cc 3 & 0 \\
2 & \cc -2 & \cc 1 & \cc 0 & 1 \\
2 & \cc 1 & \cc -3 & \cc 1 & 2 \\
 & 1 & 3 & 0 & 
\end{array}\right]
\longrightarrow
\left[\begin{array}{ccccc}
 & 4 & 3 & 0 &  \\
1 & \cc 0 & \cc 0 & \cc 3 & 0 \\
4 & \cc -2 & \cc 1 & \cc 0 & 1 \\
3 & \cc 1 & \cc -3 & \cc 1 & 2 \\
 & 1 & 3 & 0 & 
\end{array}\right].$$
\end{example}

So, half of the equalize part of this step (i.e., on rows and columns going through the patch) is done after these (or equivalent) changes have been made.
An important part of the rest of it will be done with the spread of mass in all four directions around the patch. The spread part will take care of the fact that the new rows and columns (i.e., those that do not go through the patch) may have total mass larger than $\mu$, and is done as follows.
(1) \emph{Negative horizontal:} We go column by column, moving to the left, away from the patch.
If $\sum_{j=1}^{m} V(R_{0,j})\leqslant\mu$, we are done with spreading in this direction.
If not, there exists an $s\in[m]$ such that $$\displaystyle\nu=\sum_{j=1}^{s-1} V(R_{0,j})\leqslant\mu<\sum_{j=1}^{s} V(R_{0,j})$$
Here, we adopt an agreement that summation oven an empty set results in a zero sum.
In this case we introduce a new column to the left and we set \begin{enumerate}[(a)]
\item $V(R_{-1,j})=0$ and leave $V(R_{0,j})$ unchanged for $j\in[m],j<s$,
\item $V(R_{-1,s})=V(R_{0,s})+\nu-\mu$ and $V(R_{0,s})=\mu -\nu$, and
\item $V(R_{-1,j})=V(R_{0,j})$ and $V(R_{0,j})=0$ for $j>s$.
\end{enumerate}
This way the volume of each row remains the same but the part of the volume of the $0^\mathrm{th}$ column that exceeds $\mu$ moves to the $-1^\mathrm{st}$ column.
We proceed inductively adding column after column to the left and moving the rest of the mass further on until we reach the situation that all the vertical columns to the left of the patch have volume equal to $\mu$ except possibly the leftmost one, the volume of which may be somewhat smaller.
Now, we proceed by adjusting this algorithm in an obvious way in the three remaining directions as before:
(2) \emph{Positive horizontal}, (3) \emph{Negative vertical}, (4) \emph{Positive vertical}.

\begin{example}\label{ex:spread}
For the mass patter in Example~\ref{ex:fork}, we only need to spread mass west and north, and the spread step results in
$$\left[\begin{array}{ccccc}
 & 4 & 3 & 0 &  \\
1 & \cc 0 & \cc 0 & \cc 3 & 0 \\
4 & \cc -2 & \cc 1 & \cc 0 & 1 \\
3 & \cc 1 & \cc -3 & \cc 1 & 2 \\
 & 1 & 3 & 0 & 
\end{array}\right]
\longrightarrow
\raisebox{0.2cm}{
$\left[\begin{array}{cccccc}
 &  & 0 & 3 & 0 &  \\
 &  & 4 & 0 & 0 &  \\
1 & 0 & \cc 0 & \cc 0 & \cc 3 & 0 \\
3 & 1 & \cc -2 & \cc 1 & \cc 0 & 1 \\
0 & 3 & \cc 1 & \cc -3 & \cc 1 & 2 \\
 &  & 1 & 3 & 0 & 
\end{array}\right]$}.
$$
\end{example}

The final step of the algorithm depends on how the equalize and spread step resolves. The most optimal case, when all columns and rows end up with total mass precisely $\mu$, is considered in the following example.

\begin{example}\label{ex:V0} Consider the case that $V_0=0$ and that all rows and columns added in the equalize and spread step of our algorithm have the same total mass $\mu$. Furthermore, assume that the algorithm added $m'$ new (nonempty) rows and $n'$ new (nonempty) columns (note that it can happen that, say, all $x_j^-$ are zero, in which case we have not added any mass in the west column).
Let us compute the volume of the horizontal band section (exactly) containing the patch. On one hand, it contains $m$ rows of mass $\mu$, so that its total mass equals $m\mu$.
On the other hand it contains the central mass zero together with $n'$ added columns of mass $\mu$ which results in $n'\mu$.
So, we get $n'=m$ and through the volume of the vertical band section we conclude similarly that $m'=n$.
Consequently, our procedure results in a square region, that we can now rescale to obtain a mass distribution of a quasi-copula: divide the sides of the obtained square $(m'+n)\times (n'+m)$ by $N=m'+n=n'+m$ and move it so that the southwest corner comes into the origin, and divide all the masses by $N\mu$ to get the total mass to $1$.
\end{example}

In general, the implementation of the final step of our algorithm will depend on the starting mass of the patch $V_0$ and the forking of mass when determining the ``new'' values of $\mu^{h-}_j,\mu^{h+}_j$ and $\mu^{v-}_i, \mu^{v+}_i$ denoted respectively by $x^{-}_j,x^{+}_j$ and $y^{-}_i, y^{+}_i$ for all $i\in[m]$ and $j\in[n]$.
Denote by $\rho$ the remainder of the division of $V_0$ by $\mu$, i.e.,  $\rho=V_0-\mu\left\lfloor \frac{V_0}{\mu}\right\rfloor$, and observe that $0\leqslant \rho<\mu$.
Let us first consider a special case of successful forking which will help us understand the general case. We think of $x^{-}_j,x^{+}_j$ and $y^{-}_i, y^{+}_i$ for all $i\in[m]$ and $j\in[n]$ as $2(m+n)$ unknowns subject to constraints
\begin{alignat*}{2}
   \mu^{h-}_j  & \leqslant x_j^-,\quad  \mu^{h+}_j&\leqslant x_j^+,\quad x_j^-+\sum_{i=1}^{m}V(R_{ij})+x_j^{+}=&\mu\ \ \mbox{for}\ \ j\in[n]\\
   \mu^{v-}_i  & \leqslant y_i^-,\quad \mu^{v+}_i&\leqslant y_i^+,\quad   y_i^-+\sum_{j=1}^{n}V(R_{ij})+y_i^{+}=&\mu\ \ \mbox{for}\ \ i\in[m].
\end{alignat*}
It is possible to express half of the unknowns and replace each pair of inequalities by a double inequality:
\begin{equation*}
  \mu^{h-}_j  \leqslant x_j^-
  \leqslant \mu-\mu_j^{h+}-\sum_{i=1}^{m}V(R_{ij})
  \quad\mbox{and}\quad\mu^{v-}_i  \leqslant y_i^-\leqslant \mu-\mu_i^{v+}-\sum_{j=1}^{n}V(R_{ij}).
\end{equation*}
The special case is defined by the following two conditions: \textbf{(A)} $\rho>0$, and \textbf{(B)} there exist solutions of the inequalities above such that sums $\sum_{j=1}^{n}x_j^-$ and $\sum_{i=1}^{m}y_i^-$ are both integer multiples of $\mu$. Clearly, Condition \textbf{(B)} implies that the columns appended in the spreading phase of the algorithm in the negative horizontal direction each contains mass $\mu$. The same holds for the negative vertical direction. Now, the total mass added in the positive horizontal direction equals
\[
    \sum_{j=1}^{n}x_j^+=n\mu-\sum_{j=1}^{n}x_j^--V_0=\hat{n} \mu -\rho=(\hat{n}-1)\mu+(\mu-\rho),
\]
for some integer $\hat{n} \geqslant 1$. So, in the spreading phase our algorithm will stop with a column of mass $\mu-\rho$. The same conclusion holds for the positive vertical direction. In order to finish the algorithm we only need to add mass $\rho$ which is positive by Condition \textbf{(A)} in the very northeast square of our extended region. So, we have proven most of Theorem \ref{final} for this case.

\begin{example}\label{ex:final}
For the mass pattern in Example~\ref{ex:spread} we have $\mu=4$, $\rho=1$, and
\begin{align*}
x^-_1 &=3, &x^+_1 &=2, &y^-_1 &=1, &y^+_1 &=4, \\
x^-_2 &=4, &x^+_2 &=1, &y^-_2 &=3, &y^+_2 &=3, \\
x^-_3 &=1, &x^+_3 &=0, &y^-_3 &=0, &y^+_3 &=0, 
\end{align*}
so that $\sum_{j=1}^{n}x_j^-=8$ and $\sum_{i=1}^{m}y_i^-=4$ are both integer multiples of $\mu$ as in the special case above (i.e. the spreading resulted in full east and south column).
Thus, the final step of the algorithm gives
$$\left[\begin{array}{cccccc}
 &  & 0 & 3 & 0 &  \\
 &  & 4 & 0 & 0 &  \\
1 & 0 & \cc 0 & \cc 0 & \cc 3 & 0 \\
3 & 1 & \cc -2 & \cc 1 & \cc 0 & 1 \\
0 & 3 & \cc 1 & \cc -3 & \cc 1 & 2 \\
 &  & 1 & 3 & 0 & 
\end{array}\right]
\raisebox{-0.2cm}{$\longrightarrow$}
\left[\begin{array}{cccccc}
 &  & 0 & 3 & 0 & \cc 1 \\
 &  & 4 & 0 & 0 &  \\
1 & 0 & \cc 0 & \cc 0 & \cc 3 & 0 \\
3 & 1 & \cc -2 & \cc 1 & \cc 0 & 1 \\
0 & 3 & \cc 1 & \cc -3 & \cc 1 & 2 \\
 &  & 1 & 3 & 0 & 
\end{array}\right],$$
and we fill the rest with zeros.
\end{example}

We are now in position to treat the rest of the cases depending on the result after forking and spreading. As seen before, a solution for forking always exists. Denote  $\displaystyle S_h^-= \sum_{j=1}^{n}x_j^-$, $\displaystyle S_h^+= \sum_{j=1}^{n}x_j^+$, $\displaystyle S_v^-= \sum_{i=1}^{m}y_i^-$, and $\displaystyle S_v^+= \sum_{i=1}^{m}y_i^+$.
We can assume that all four sums are strictly positive (so that none of the outer rows and columns is empty), since the only case when this cannot be achieved, is the case when the Geometric condition is already satisfied, and there is no mass to fork, i.e., all rows and columns already have the same mass.
In this case we just leave the patch as is.
Each of the following five conditions \textbf{(C1)} $S_h^-=\mu\left\lfloor \dfrac{S_h^-}{\mu}\right\rfloor$,  \textbf{(C2)} $S_h^+=\mu\left\lfloor \dfrac{S_h^+}{\mu}\right\rfloor$,  \textbf{(C3)} $S_v^-=\mu\left\lfloor \dfrac{S_v^-}{\mu}\right\rfloor$,  \textbf{(C4)} $S_v^+=\mu\left\lfloor \dfrac{S_v^+}{\mu} \right\rfloor$, and \textbf{(C5)} $\rho=0$, can be either true or false which amounts to 32 cases.
Fortunately, the solution of a case can easily be implemented to the case in which we exchange directions in either horizontal or vertical dimension or both. For instance, the above solution of the special case gives rise to solutions in three more cases amounting to mass $\rho$ ending up in either the very northwest, southwest, or southeast corner. Also, some of the 32 cases are not feasible.

Denote by $\rho_h^-$, $\rho_h^+$, $\rho_v^-$, $\rho_v^+$ the respective remainders of the division of $S_h^-$, $S_h^+$, $S_v^-$, $S_v^+$ by $\mu$.
Let us compute the volume of the horizontal band exactly containing the patch. On one hand, when computing sums of rows, it equals $n\mu$, on the other hand, if we compute the sums of columns, we get $S_h^-+V_0+S_h^+$, which equals $\rho_h^-+\rho+\rho_h^+$ modulo $\mu$.
Since this must be divisible by $\mu$, we are left with only two options: either $\rho_h^-+\rho_h^+=\mu-\rho$ or $\rho_h^-+\rho_h^+=2\mu-\rho$. Similarly, for the vertical direction, we have either
$\rho_v^-+\rho_v^+=\mu-\rho$ or $\rho_v^-+\rho_v^+=2\mu-\rho$.

Assume first that conditions \textbf{(C1)}--\textbf{(C4)} are all false, i.e., $\rho_h^-$, $\rho_h^+$, $\rho_v^-$, $\rho_v^+$ are all positive, and denote by
$\gamma_{--},\gamma_{-+},\gamma_{++},$ and $\gamma_{+-}$ the tentative masses of the southwest, southeast, northeast, and northwest corner of the extended region.
Consider the case when $\rho_h^-+\rho_h^+=\mu-\rho$ and $\rho_v^-+\rho_v^+=\mu-\rho$.
If we set $\gamma_{--}=\min\{\rho_h^+,\rho_v^+\}$, we get necessarily $\gamma_{+-}=\mu -\rho_h^--\min\{\rho_h^+,\rho_v^+\}\geqslant\mu -\rho_h^--\rho_h^+=\rho \geqslant 0$, and similarly, $\gamma_{-+}=\mu -\rho_v^--\min\{\rho_h^+,\rho_v^+\} \geqslant 0$. Clearly, both masses are also smaller than $\mu$.
Now, the mass in the last corner can be computed in two ways, both giving the same result $\gamma_{++}=\rho_h^- -\rho_v^+ +\min\{\rho_h^+,\rho_v^+\}= \rho_v^- -\rho_h^+ +\min\{\rho_h^+,\rho_v^+\}$ as a short computation reveals.
It is also easy to see that the value of the mass in this corner equals $\min\{\rho_h^-,\rho_v^-\}$ thus resulting in a solution of our problem with the desired properties.

Next consider the case when $\rho_h^-+\rho_h^+=2\mu-\rho$ and $\rho_v^-+\rho_v^+=2\mu-\rho$.
In order to find the solution in this case let us set $\gamma_{--}=0$, then $\gamma_{+-}=\mu -\rho_h^-$ and $\gamma_{-+}=\mu  -\rho_v^-$ and all these values belong to $[0,\mu)$. The fourth corner brings us to the same value either way
$\gamma_{++}=\rho_h^--\rho_v^+=\rho_v^--\rho_h^+$.
If the obtained value is positive we are done since it is clearly smaller than $\mu$.
If not, we start by letting $\gamma_{++}=0$, get  $\gamma_{+-}=\mu -\rho_v^+$ and $\gamma_{-+}=\mu  -\rho_h^+$, and finally $\gamma_{--}=\rho_v^+-\rho_h^-=\rho_h^+-\rho_v^-\geqslant0$ and we are done again.

Finally, consider the case when, say, $\rho_h^-+\rho_h^+=\mu-\rho$ and $\rho_v^-+\rho_v^+=2\mu-\rho$.
In this case we need to add an additional row, say, to the north of the extended region. Denote the tentative masses of the west and east square of this row by $\gamma'_{+-}$ and $\gamma'_{++}$.
If we set $\gamma_{--}=\gamma_{+-}=0$, then $\gamma_{-+}=\mu-\rho_v^-$, $\gamma_{++}=\mu-\rho_v^+$, and $\gamma'_{+-}=\mu-\rho_h^-$.
We can calculate $\gamma'_{++}$ in two ways giving the same value $\gamma'_{++}=\rho_h^-=\mu-(\mu-\rho_v^+)-\rho_h^--(\mu-\rho_v^-)$. Note that all these values are positive.

\begin{example}
We give an example of a signed mass pattern and consequent forking to the north, that results in the situation where an additional row needs to be added.
\begin{align*}
\left[\begin{array}{cccccc}
1 & -2 & 4 & -1
\end{array}\right]
&\longrightarrow
\left[\begin{array}{cccccc}
 & 0 & 2 & 0 & 1 & \\
1 & \cc 1 & \cc -2 & \cc 4 & \cc -1 & 1 \\
 & 0 & 2 & 0 & 1 & \\
\end{array}\right]
\longrightarrow
\left[\begin{array}{cccccc}
 & 3 & 4 & 0 & 4 & \\
1 & \cc 1 & \cc -2 & \cc 4 & \cc -1 & 1 \\
 & 0 & 2 & 0 & 1 & \\
\end{array}\right]\\
&\longrightarrow
\raisebox{0.4cm}{$\left[\begin{array}{cccccc}
 & 0 & 0 & 0 & 3 & \\
 & 0 & 3 & 0 & 1 & \\
 & 3 & 1 & 0 & 0 & \\
1 & \cc 1 & \cc -2 & \cc 4 & \cc -1 & 1 \\
 & 0 & 2 & 0 & 1 & \\
\end{array}\right]$}
\longrightarrow
\raisebox{0.6cm}{$\left[\begin{array}{cccccc}
\cc 3 &  &  &  &  & \cc 1\\
 & 0 & 0 & 0 & 3 & \cc 1\\
 & 0 & 3 & 0 & 1 & \\
 & 3 & 1 & 0 & 0 & \\
1 & \cc 1 & \cc -2 & \cc 4 & \cc -1 & 1 \\
 & 0 & 2 & 0 & 1 & \cc 1\\
\end{array}\right]$}.
\end{align*}
Note, however, that if we forked the mass differently we could have ended in one of the previous cases, where no additional row is needed. For example
\begin{align*}
\left[\begin{array}{cccccc}
1 & -2 & 4 & -1
\end{array}\right]
&\longrightarrow
\left[\begin{array}{cccccc}
 & 0 & 2 & 0 & 1 & \\
1 & \cc 1 & \cc -2 & \cc 4 & \cc -1 & 1 \\
 & 0 & 2 & 0 & 1 & \\
\end{array}\right]
\longrightarrow
\left[\begin{array}{cccccc}
 & 1 & 4 & 0 & 4 & \\
1 & \cc 1 & \cc -2 & \cc 4 & \cc -1 & 1 \\
 & 2 & 2 & 0 & 1 & \\
\end{array}\right]\\
&\longrightarrow
\raisebox{0.2cm}{$\left[\begin{array}{cccccc}
 & 0 & 0 & 0 & 1 & \\
 & 0 & 1 & 0 & 3 & \\
 & 1 & 3 & 0 & 0 & \\
1 & \cc 1 & \cc -2 & \cc 4 & \cc -1 & 1 \\
 & 2 & 2 & 0 & 0 & \\
 & 0 & 0 & 0 & 1 & \\
\end{array}\right]$}
\longrightarrow
\raisebox{0.2cm}{$\left[\begin{array}{cccccc}
\cc 2 & 0 & 0 & 0 & 1 & \cc 1\\
 & 0 & 1 & 0 & 3 & \\
 & 1 & 3 & 0 & 0 & \\
1 & \cc 1 & \cc -2 & \cc 4 & \cc -1 & 1 \\
 & 2 & 2 & 0 & 0 & \\
\cc 1 & 0 & 0 & 0 & 1 & \cc 2\\
\end{array}\right]$}.
\end{align*}
\end{example}

Next, assume that only one of the conditions \textbf{(C1)}--\textbf{(C4)} is true and that the other conditions are false,
say $\rho_h^-=0$, $\rho_h^+,\rho_v^-,\rho_v^+>0$. Then automatically $\rho>0$, since $\rho_h^-+\rho+\rho_h^+$ is divisible by $\mu$. Furthermore, $\rho_h^-+\rho_h^+=\mu-\rho$, since the left side is $\rho_h^+<\mu$.
So we are left with two cases depending on the value of $\rho_v^-+\rho+\rho_v^+$.
Assume first that $\rho_v^-+\rho_v^+=2\mu-\rho$.
Since $\rho_h^-=0$, this means that the west column is full or empty and we must put $\gamma_{--}=\gamma_{+-}=0$ (if the column is empty, we remove it, since it will not be part of the final region). Consequently,
$\gamma_{-+}=\mu-\rho_v^-$ and $\gamma_{++}=\mu-\rho_v^+$.
Then the sum of the right column is $2\mu-\rho_v^--\rho_v^++\rho_h^+=\mu$, as needed.

Now, assume that $\rho_v^-+\rho_v^+=\mu-\rho$. In this case we need to add an additional column, say to the east of the extended region. Denote the tentative masses of its south and north square by $\gamma'_{-+}$ and $\gamma'_{++}$.
We again have $\gamma_{--}=\gamma_{+-}=0$, and we put additionally $\gamma_{-+}=0$.
Then we have $\gamma'_{-+}=\mu-\rho_v^-$, $\gamma_{++}=\mu-\rho_h^+$, and $\gamma'_{++}=\rho_h^+-\rho_v^+ \geqslant 0$, since $\rho_h^+=\rho_h^-+\rho_h^+=\mu-\rho=\rho_v^-+\rho_v^+$ by assumption.
Note that $\gamma'_{-+}+\gamma'_{++}=\mu+\rho_h^+-\rho_v^+-\rho_v^-=\mu$ as needed.

Observe that as long as the condition \textbf{(C5)} is false we cannot have more than one zero remainder in either of the two directions, i.e., conditions \textbf{(C1)} and \textbf{(C2)} cannot simultaneously be true and conditions \textbf{(C3)} and \textbf{(C4)} cannot simultaneously be true.
Thus, the only remaining case under the assumption $\rho>0$ is the special case solved in the very beginning. So, we can assume from now on that $\rho=0$.
In both directions we now have either a zero remainder on each side, or a pair $\rho_h^->0$ and $\rho_h^+=\mu-\rho_h^->0$, and similarly for the vertical direction.
If both directions have zero sides, we arrive at the solution described in the Example~\ref{ex:V0}.
If both directions have nonzero sides, we are in the case already considered above.
So the last case to consider is the case when one direction has zero sides and the other one does not, say $\rho_h^-=\rho_h^+=0$ and $\rho_v^-,\rho_v^+>0$.
In this case both east and west columns are full or empty (we remove the column if it is empty), and $\rho_v^-+\rho_v^+=\mu$, since this is positive and must be divisible by $\mu$.
We need to add an additional column, say to the east of the extended region and put mass $\mu-\rho_v^+$ into its north square and mass $\mu-\rho_v^-$ into its south square. The total mass of this column is then equal to $\mu$, as needed.

We have thus added mass in the extended region so that all added squares have nonnegative mass and the total mass of every row and column is $\mu$. 
If $N$ is the total number of rows and $M$ the total number of columns, then the total mass in the extended region is $N\mu=M\mu$, counting by rows and columns, so that $N=M$ and the region is a square.
In addition, we have added enough positive mass in each row and column to fulfill the Geometric condition so that the resulting mass pattern will correspond to a quasi-copula after normalization.
This finishes the proof of the following theorem.

\begin{theorem}\label{final}
  Let a signed mass pattern on a patch $\PP=[0,m] \times [0,n]$ be given.
  \begin{enumerate}[(a)]
    \item There exist integers $h^- \leqslant 0$, $v^-\leqslant0$, $h^+\geqslant m$ and $v^+\geqslant n$ such that $h^+-h^-=v^+-v^-=:N$, and such that the big square $\mathds{G}=\{(x,y) \mid h^- \leqslant x \leqslant h^+,\ v^- \leqslant y \leqslant v^+\}$ consists of $N^2$ small squares $R_{ij}$ with side length $1$.
    \item Moreover, there exists a spread of masses on the small squares $V(R_{ij})$ with $h^- \leqslant i \leqslant h^+$ and $v^- \leqslant j \leqslant v^+$, such that $$\sum_{k=v^-}^{v^+}V(R_{ik})=\sum_{k=h^-}^{h^+} V(R_{kj})=:\mu \quad \text{for all $i$ and $j$},$$
    and the masses of squares in the patch $\PP$ are equal to the masses prescribed by the signed mass pattern. 
    \item In addition, if we divide the sides of the big square $\mathds{G}$ by $N$, move its southwest corner to the origin, and divide all the masses by $N\mu$, we get a mass distribution of some discrete quasi-copula.
  \end{enumerate}
\end{theorem}

\ns{Let us conclude this section with a few remarks. As already pointed out, the value $\mu$ in the algorithm is optimized to be as small as possible, which implies that the eventual size of the patch in the final quasi-copula (after rescaling) is optimized to be as large as possible.
The main takeaway is that any signed mass pattern can be rescaled and extended to a mass distribution of some discrete quasi-copula. The resulting quasi-copula is not unique, since we could spread the additional mass in many other different ways, and we could even add some new negative mass if we wanted. If the starting mass pattern contains only positive mass then our algorithm produces a discrete copula, since we only added positive mass.}

\section{Concluding remarks}

\ns{
In the paper we presented several novel methods for obtaining quasi-copulas, which demonstrate whether we have more or less freedom in constructing quasi-copulas versus similar methods for constructing copulas. There are several other constructions of quasi-copulas available in the literature. For example, multivariate Archimedean quasi-copulas and super-modular quasi-copulas were considered in \cite{ArGaMeDeBa2}, while multivariate Bertino quasi-copulas are investigated in \cite{ArGaMeDeBa3}. For some other families of quasi-copulas we refer the reader to \cite{ArGaMeDeBa}.
}

In Section~\ref{sec:mot} we have shown that there are no other quasi-copulas then those that we obtain by taking infima and suprema of copulas. So every multivariate quasi-copula can be constructed using sets of copulas. Now, the set of quasi-copulas is richer than the set of copulas, but more restricted for shuffling. The deep reason behind this phenomenon observed here may be the fact that quasi-copulas do not induce a $\sigma$-additive signed measure in general \cite{NeQuMoRoLaUbFl3}. Nevertheless, when considering approximations with shuffles in Sections~\ref{sec:shuff} and \ref{sec:gen} in a sophisticated way, it turned out that we still have enough freedom to get interesting approximation results.

To conclude:
The bivariate method presented in Section 5 for constructing a quasi-copula from a given mass pattern is useful particularly for constructing examples and counterexamples of quasi-copulas. An extension of this method to the case of multivariate quasi-copulas is one of the major problems proposed in this paper. Another important incentive of ours is our novel approach to tensor product of copulas given in Section \ref{sec:gen}. But perhaps the most significant challenge introduced here are the shuffles of quasi-copulas.

\bibliographystyle{amsplain}

\end{document}